%% file: gmrf_paper.tex
\input{header.tex}
\numberwithin{equation}{section}
\numberwithin{lemma}{section}
\title{Approximating Gaussian Whittle-Mat\'ern Fields over Well-Centered Triangulations of Riemannian Manifolds}
\author{Srinivas Nambirajan}
\date{June 10, 2026}
\begin{document}
\maketitle

\abstract{Markovian Whittle-Mat\'ern fields have been convergently approximated by discrete Gauss Markov Random Fields (GMRFs) with sparse precision matrices using a Finite Element approximation of the two-parameter family, 
\[ (\grk^2 - \lap)^{\gra/2} u = \c{W}, \;\; \grk \in \real, \; \gra \in \nat,  \]
of SPDEs.  Using recent developements in the analysis of Discrete Exterior Calculus (DEC), we present a different, yet closely related, convergent GMRF approximation to these Mat\'ern fields over complete boundaryless Riemannian manifolds of any dimension discretized as \emph{well-centered} simplicial complexes. This convergent method (i) is \emph{agnostic} to $\gra, \grk$ and thus allows a \emph{universal approximation scheme} for the precision and covariance matrices of the entire $(\gra, \grk)$-family of GMRFs, so they may be inferred rather than guessed. (ii) inherently models pointwise and piecewise-smoothed measurements of a random field and approximates both equally well (iii) is computationally independent of the interpolants used - it suffers no overhead if one convergent interpolant were replaced with another suitable interpolant over the same mesh. Furthermore, we show that, on discretizations that are well-connected in a precise sense, and \emph{volume-concentrated}, the precision matrices are spectral functions of a graph-laplacian. We provide a low rank approximator to the family of such Mat\'ern GMRFs and mention a use case: reducing the number of measurements needed to model the GMRF by compressed-sensing.}

\input{S.Intro}

\input{S.AbstractSetup}

\input{S.ECTools}

\input{S.Approximation}

\input{S.Results}

\input{S.FutureWork}

\input{S.Credits}

\bibliography{thebib}
\bibliographystyle{siam}
\end{document}

%% file: header.tex
\documentclass[]{article} 
\usepackage[mathcal]{eucal}
\usepackage{mathrsfs}
\usepackage{amsmath}
\usepackage{amssymb}
\usepackage{MnSymbol} 

\usepackage[a4paper,
		bindingoffset=0.15in,
		left=0.8in,
		right=0.8in,
		top=0.6in,
		bottom=0.6in,
	footskip=0.25in]{geometry}

\usepackage{amsthm}
\newtheorem{thm}{Theorem}
\newtheorem{lemma}{Lemma}
\newtheorem{prop}{Proposition}
\newtheorem{corl}{Corollary}
\newtheorem{defn}{Definition}
\newtheorem{assume}{Assumption}
\newtheorem{remark}{Remark}
\newtheorem{notation}{Notation}
\usepackage{graphicx}

\usepackage{tcolorbox} 
\usepackage{enumitem}
\usepackage{xcolor}  

\renewcommand{\comment}[1]{\hfill\textcolor{gray}{\#\ \textit{#1}}}

\newenvironment{minimalgo}[1]{%
  \begin{tcolorbox}[
    colback=gray!0.1,          
    colframe=gray!50,       
    boxrule=0.2pt,          
    arc=0pt,                
    boxsep=5pt,
    left=6pt, right=6pt, top=3pt, bottom=3pt 
  ]
  \noindent\textbf{Algorithm:} #1\\[0.8em]
}{%
  \vspace{0.8em}\noindent\textbf{end}
  \end{tcolorbox}
}

\usepackage{hyperref}

\renewcommand{\(}{\ensuremath{\left( \right.}}
\renewcommand{\)}{\ensuremath{\left. \right)}}
\newcommand{\<}{\ensuremath{\left \langle \right.}}
\renewcommand{\>}{\ensuremath{\left. \right \rangle}}

\newcommand{\seminorm}[1]{\ensuremath{ \left| #1 \right|}}
\newcommand{\abs}[1]{\ensuremath{\seminorm{#1}}}

\newcommand{\norm}[1]{\seminorm{\seminorm{#1}}}

\newcommand{\Gt}{\succ}

\newcommand{\Leq}{\preceq}

\newcommand{\gra}{\alpha}

\newcommand{\grd}{\delta}
\newcommand{\gre}{\epsilon}
\newcommand{\grf}{\phi}

\newcommand{\grk}{\kappa}
\newcommand{\grl}{\lambda}
\newcommand{\grm}{\mu}

\newcommand{\grp}{\pi}

\newcommand{\grr}{\rho}
\newcommand{\grs}{\sigma}
\newcommand{\grt}{\tau}

\newcommand{\grw}{\omega}

\newcommand{\grL}{\Lambda}

\newcommand{\grP}{\Pi}

\newcommand{\grS}{\Sigma}

\newcommand{\grW}{\Omega}

\newcommand{\from}{\leftarrow}

\newcommand{\f}[1]{\ensuremath{\mathscr{#1}}}
\renewcommand{\c}[1]{\ensuremath{\mathcal{#1}}}

\newcommand{\Prob}{\ensuremath{\mathbb{P}}}

\newcommand{\mean}[1]{\ensuremath{\mathbb{E}\left[ #1 \right]}}

\newcommand{\man}[1]{\ensuremath{\mathcal{#1}}}

\newcommand{\lap}{\ensuremath{\mathbf{\Delta}}}

\newcommand{\dlap}{\ensuremath{\mathtt{\Delta}}}

\newcommand{\dd}{\ensuremath{\mathtt{d}}}
\newcommand{\dad}{\ensuremath{\mathtt{\delta}}}
\newcommand{\dstar}{\ensuremath{\smallstar}}
\newcommand{\astar}{\ensuremath{\ast}}

\DeclareMathOperator{\Range}{Range}

\DeclareMathOperator{\Dim}{Dim}

\DeclareMathOperator{\Div}{div}
\DeclareMathOperator{\Grad}{grad}
\DeclareMathOperator{\Diagonal}{Diagonal}

\DeclareMathOperator{\Cov}{Cov}
\newcommand{\flange}[1]{\ensuremath{\boxslash_{#1}}} 
\newcommand{\real}{\ensuremath{\mathbb{R}}}
\newcommand{\reals}[1]{\ensuremath{\mathbb{R}^{#1}}}
\newcommand{\nat}{\ensuremath{\mathbb{N}}}

\newcommand{\rats}{\ensuremath{\mathbb{Q}}}

\newcommand{\heq}{\ensuremath{\stackrel{h}{\equiv}}}

\newcommand{\inv}[1]{\ensuremath{{#1}^{-1}}}

\newcommand{\floor}[1]{\ensuremath{\lfloor #1 \rfloor}}
\newcommand{\ceil}[1]{\ensuremath{\lceil #1 \rceil}}

%% file: S.Intro.tex
\section{Introduction}

Let $\lap:= \Div \circ \Grad$ be the usual laplacian over $H^2 \(\reals{d}\)$ and let $\c{W}$ be the white-noise process over $\reals{d}$. The solutions, $u$, to the Stochasic Partial Differential Equation (SPDE)
\begin{equation}\label{eqn:basic_spde}
(\grk^2 - \lap)^{\gra/2} u = \c{W}, \;\; \grk \in \real, \; \gra \in \reals{+} 
\end{equation} 
are known to be gaussian processes, with $\gra > d/2$ yielding processes having the Whittle-Mat\'ern covariance \cite{whittle1963}, and $\gra \in \nat$ being Markovian \cite{Rozanov77}.  
This may be extended to complete manifolds without boundary as in \cite{rue11} by replacing the positive semidefinite $(- \lap)$ in (\ref{eqn:basic_spde}) with the positive semidefinite \emph{Hodge-Laplacian} on 0-forms (equivalently the \emph{Laplace-Beltrami}), $\lap_0:= \grd d: H^2 \grL^0 (\man{M}) \to C\grL^0(\man{M})$, over a riemannian manifold, $\man{M}$, of dimension, $m$, to obtain
\begin{equation}\label{eqn:man0_spde} 
(\grk^2 + \lap_0)^{\gra/2} u = \c{W}, \;\; \grk \in \real, \; \gra \in \nat+m/2
\end{equation}  
with $u$ now a GMRF with Mat\'ern covariance over $\man{M}$.
In their foundational work in statistics and computation \cite{rue11} obtain a convergent, finite-dimensional approximation, $\Cov(\tilde{u})$ of $\Cov(u)$ for $u$ solving (\ref{eqn:man0_spde}) for $\gra \in \nat + d/2$ over a sphere using finite element techniques and establish the link to discrete Gauss-Markov Random Fields, as the coefficients, $x$, of $\tilde{u}$ form a GMRF with a discrete covariance matrix. The work of \cite{rue11, rue22} has been instrumental in a variety of fields \cite{bachl19inlabru, KellerLindstrom14, OzoneBolinLindgren11, OceanTempDinsdale19, MalariaBhatt15, Konstantinoudis20DiseaseModeling, Cameletti2013Spatio, Fuglstad15} from kriging to disease-modelling, due to computational ease \cite{KellerLindstrom14, RueSimpsonThink12} of obtaining the covariance of a sparse-precision GMRF \cite{RueINLA09, LindgrenRINLA15} over one obtained from sampling from a continuous Gaussian field using MCMC or similar.  

The definition of (\ref{eqn:man0_spde}) may be extended to $k$-form gaussians by simply using $\lap_k$ -
\begin{equation} \label{eqn:man_spde}
(\grk^2 + \lap_k)^{\gra/2} u = \c{W}_k, \;\; \grk \in \real, \; \gra \in \nat
\end{equation}
with the $k$-form gaussian white noise, $\c{W}_k$ and $u$ being expressed in the countable eigenbasis, $\{ \psi_{k, i} \}$ of $\lap_k$ with eigenvalues, $\grl_i$, as: 
\begin{equation} \label{eqn:kkkl} 
	u = \sum \grw_i \sqrt{\grl_i} \psi_{k, i}, \;\; \c{W}_k = \sum \grw_i \psi_{k, i}  
\end{equation}
where $\grw_i \sim \c{N}(0, 1)$ form a countable collection of iid random variables in each sum. Such $k$-form gaussians have been recently posed similarlty and studied in \cite{NicoudKrause24,CaoSheffieldGausskForms26}. 

The equivalent to FEM for (\ref{eqn:man_spde}) over $k$-forms for $k\geq 0$ is Finite Element Exterior Calculus (FEEC) \cite{AFW06, AFW09, AFWBook18} which provides tools for solving such equations over discretizations of manifolds. 

More recently, Discrete Exterior Calculus \cite{Hirani03, Desbrun03} was developed which circumvents the need to work with continuous objects at all when doing certain operations over discretized manifolds. It does so by providing consistent discrete equivalents of objects from exterior calculus that are defined on \emph{well-centered} simplicial triangulations of manifolds, and allows us to pose and solve elliptic PDEs in discrete form directly over a triangulation of a manifold. This approach has naturally found a variety of applications in computational geometry \cite{VecFieldDecompGraphics03, DiscShellsGrinspunHirani03, GeoDistPredCrane17}, and has led to generalizations \cite{LHL} to graphs and clique-complexes \footnote{1-skeletons of simplicial complexes} that have informed network analysis \cite{BrainNetHodge19, SensNetTahbaz10, LinkPredHodge18}, voting and ranking \cite{HodgeRank11}, etc. DEC reconciles with the continuous by establishing consistency with FEEC via convergence arguments using certain \emph{Interpolation maps}, and was recently shown by \cite{GP26} to provide convergent solutions to the $k$-form hodge-poisson problem over riemannian manifolds of arbitrary dimension and for all\footnote{results with restrictions to manifold-dimension and/or $k$ appear in \cite{Kaibo26, TS16DEC}} relevant $k$.

\subsection{On Approximating Measurement}
In most applications \cite{OzoneBolinLindgren11, OceanTempDinsdale19, MalariaBhatt15, Konstantinoudis20DiseaseModeling, Cameletti2013Spatio, Fuglstad15} measurements of $u$ are made at finitely many points on a manifold where the galerkin-approximation to $u$ solving (\ref{eqn:basic_spde}), (\ref{eqn:man_spde}) in the mean-square sense carries a further implicit assumption: that the measurement of the \emph{approximation of} $u$ is an effective proxy for the measurement of $u$\footnote{Indeed, measurement is its own source of error and floating-point representation makes approximation unavoidable. However, this assumption forces the approximation of $u$ to be sufficiently good \emph{at the points of measurement} to ensure that other such unavoidable errors are the ones of leading order.}. Such an assumption is justified by a good approximation of $u$. An upshot of this work is the removal of this implicit assumption by taking the exact measurement-operator into account. We thus approximate the covariance of the exact measurements without having to implicitly adjust them.


\subsection{Approximation-Scheme over Approximation}
The \emph{approximation-scheme} provided here is distinguished from an \emph{approximation} since it provides a space of discrete gaussian processes simultaneously for all $\gra, \grk$ with a specified error-tolerance. From such a specification, one may either obtain a particular approximation to a GMRF covariance matrix by specifying $\gra, \grk$, or statistically infer $\gra, \grk$ from measurements of $u$.

\subsection{Contribution}
We provide convergent, discrete \emph{approximation-schemes} to the covariances of 
\begin{enumerate}[label=(\alph*)]
	\item continuous gauss markov fields with Whittle-Mat\'ern covariance (\ref{eqn:man0_spde}) over $\man{M}$,
	\item continuous gauss markov fields whith Whittle-Mat\'ern covariance smoothed/averaged over certain triangulations of $\man{M}$
\end{enumerate}
such that:
\begin{enumerate}
	\item The process of obtaining the approximant is agnostic to $\gra, \grk$. So, rather than requiring the parameters, $\gra, \grk$, to obtain an approximation, they may be inferred from the approximation.
	\item They are computationally efficient; we obtain a single approximation operator for all $\gra, \grk$ in $O((kn_1 + k^2 n_0)/\sqrt{\bar{\gre}})$ where $n_0, n_1$ are the counts of nodes and edges in the mesh. Even this running time is a bit pessimistic in some cases - such as inferring $\gra, \grk$.
	\item They solve the problem identified in \cite{rue11} of obtaining convergent approximations by replacing the consistent mass-matrix in their formulation by the diagonal, lumped mass-matrix; \cite{rue11} use the quadrature-error approximation of Lemma 1 in \cite{ChenThom85} to partially justify such a replacement, with the rest of the justification being heuristic. Our approximation-schemes analytically justify such a replacement in its entirety.
	\item In applications where measurements of the Whittle-Mat\'ern process are made at certain locations on the manifold (e.g. kriging), the approximate covariance is that of the \emph{actual} measurements, as the De Rham map (\S\ref{sec:ectools}) is a model of a measurement operator.	
\end{enumerate}

Our analytical framework in making these contributions is consistent with the approach of Virtual Element Method (VEM) \cite{VEMBook22}, validating any observed similarities between DEC, VEM.

\subsection{Organization of Content}

We present the weak-solution scheme of (\ref{eqn:man_spde}) in the abstract and state our main result in \S\ref{sec:protoscheme}. We develop and present the necessary tools from Discrete Exterior Calculus and briefly summarize the recent proof of convergence of the DEC Hodge Poisson problem in \S\ref{sec:ectools}. We present the conditions of our approximation scheme and prove our result in \S\ref{sec:Approx}, followed by numerical results in \S\ref{sec:numResults}.

%% file: S.AbstractSetup.tex
\section{An Abstract Galerkin Argument} \label{sec:protoscheme}
We present the general scheme for solving polynomials in the hodge-laplacian in the mean-square sense, along the lines of \cite{rue11}. 
\begin{notation} \label{nota:basicManifold}
Let $\man{M}$ be a boundaryless uniformly $m$-dimensional riemannian manifold with metric, $\grm$ admitting a finite simplicial triangulation. Let $M$ be a finite simplicial triangulation of $\man{M}$. $M_h$ shall explicitly denote a simplicial triangulation with size-factor, $h$, when contextually necessary; this is considered implicit in $M$.

Let $\f{U}_k$ be an interpolant onto $C \grL^k (M)$, the space of continuous $k$-forms over the piecewise-flat \footnote{also called a \emph{regge}} $M$; and $\f{V}_k$ be an interpolant onto $\grL^k (M)$, the space of all $k$-forms over $M$.

	Let $p_r(\lap_k): H^{2r} \grL^k (\man{M}) \to C \grL (\man{M})$ be a strictly positive polynomial of integer order, $r$, in $\lap_k$, the hodge-laplacian on $k$-forms. 
\end{notation}

When solving 
 \begin{equation}\label{eq:polylap_fractional} 
	 (p_r (\lap_k))^{\gra/2} u = \grf, \; \; \grf \in \grL^k (\man{M}), u \in H^{2r \gra} \grL^k (\man{M}) 
\end{equation}
in the interior of $\man{M}$, we observe as in \cite{rue11} that (\ref{eq:polylap_fractional}) may be recursively described for $\gra \in 2 \nat$ with $u_{t} \in H^{2rt} \grL^k (\man{M})$ as
\begin{equation} \label{eq:polylap_recursive} 
	p_r (\lap_k) u_{t} = u_{t-2}, \;\; t \in \nat, t \leq \gra, \;\; u_0 = \grf.
\end{equation}
This helps obtain a convergent approximation, $\tilde{u}_t = \f{U}_k x_t$, to $u_t$, from an interpolant-space, $Range(\f{U}_k) \not \subset H^{2rt} \grL^k(M)$ by solving (\ref{eq:polylap_recursive}) over $M$, weakly \footnote{we choose the Galerkin-like formulation rather than the mixed-formulation that is the stable choice for general $k$ since we shall be interested only in mean-square solutions when $\grf$ is a random process for $k \in \{ 0, m\}$, and this choice is expositorily simpler.}
as 
\begin{equation} \label{eqn:polylap_recursive_weakform} 
\begin{array}{rl}
	(\f{V}_k z,  p_r (\lap_k) \f{U}_k x_t ) &= (\f{V}_k z, \f{U}_k x_{t-2}), \;\; t \in \nat, \; t \leq \gra, \\ 
	(\f{V}_k z, p_r(\lap_k) \f{U}_k x_2) &= (\f{V}_k z, \grf).\\
\end{array}
\end{equation}
The well-posedness, stability and convergence of (\ref{eqn:polylap_recursive_weakform}) depend upon the choice of $\f{V}_k, \f{U}_k$ - interpolants onto \emph{test} and \emph{trial} spaces - and $r$, as well as the quality of the discretization, $M$, of $\man{M}$ \footnote{coercivity is granted by a Poincar\'e-Friedrichs-type inequality over shape-regular $M$, boundedness is granted by the choice of $\f{U}_k, \f{V}_k$ with the Lax-Milgram lemma then guaranteeing a unique solution to (\ref{eqn:polylap_recursive_weakform})}. In \cite{AFW09}, a case of $r=1, \gra = 2$ is solved in the more stable mixed-formulation over shape-regular meshes, with interpolants, $\f{V}_k = \f{U}_k$, onto polynomial spaces that are \emph{bounded cochain projections} \cite{AFW06} of $\grL^k (M)$. 

When $\gra \in 2 \nat + 1$, however, one arrives at a fractional base-case, 
\[ (p_r (\lap_k) )^{1/2} u_1 = \grf, \]
weakly-posed over $M$ using the self-adjointness of $p_r(\lap_k)^{1/2}$ as 
\begin{equation}\label{eqn:protobase_weakform} 
(p_r (\lap_k)^{1/2} v,  p_r (\lap_k)^{1/2} u_1) = (v, p_r (\lap_k) u_1) = (v,  p_r (\lap_k)^{1/2}\grf) \;\; \forall v \in \Range(\f{V}_k), 
\end{equation}
resulting in the weak solution, $u_1 := \f{U}_k x_1, x_1 \in \reals{n_k}$, with 
\begin{equation}\label{eq:proto_base}
	\f{V}_k ^* p_r (\lap_k) \f{U}_k x_1 = \f{V}_k ^* p_r (\lap_k)^{1/2} \grf.
\end{equation}

This is especially convenient when $\grf = \c{W}_k$ is the $k$-form white noise over $\man{M}$ [cite defn] and (\ref{eq:proto_base}) is solved in the mean-square sense for the discrete covariance of $x_1$, as 
\[ \begin{array}{ll}
		\mean{x_1 x_1 ^*} &= \inv{(\f{V}_k ^* p_r (\lap_k) \f{U}_k)} \f{V}_k ^* p_r (\lap_k)^{1/2} \mean{\c{W}_k \c{W}_k ^*} p_r(\lap_k)^{1/2} \f{U}_k \inv{(\f{V}_k ^* p_r (\lap_k) \f{U}_k)} \\
		&=   \inv{(\f{V}_k ^* p_r (\lap_k) \f{U}_k)} \f{V}_k ^* p_r (\lap_k) \f{U}_k \inv{(\f{V}_k ^* p_r (\lap_k) \f{U}_k)} \quad \because \mean{\c{W}_k \c{W}_k ^*} = id \\  
	&= \inv{(\f{V}_k ^* p_r (\lap_k) \f{U}_k)}, 
\end{array} \]
whence the precision matrix is simply 
\begin{equation} \label{eqn:protobase_precision}
\inv{\mean{x_1 x_1 ^* }} =   \f{V}_k ^* p_r (\lap_k) \f{U}_k =: \f{K}_k. 
\end{equation}
This requires no approximation of the fractional operator, $p_r(\lap_k)^{1/2}$ directly.

Now, (\ref{eq:polylap_recursive}) results in
\[ \begin{array}{ll} 
		\mean{x_{t} x_{t} ^*} &= \inv{(\f{V}_k ^* p_r (\lap_k) \f{U}_k)} \f{V}_k ^* \mean{u_{t-2} u_{t-2} ^*} \f{V}_k \inv{(\f{U}_k ^* p_r (\lap_k) \f{V}_k )} \\
		&= \inv{\f{K}_k} \f{V}_k ^* \f{U}_k \mean{x_{t-2} x_{t-2} ^* }\f{U}_k ^* \f{V}_k \( \f{K}_k ^{*}\) ^{-1}. 
\end{array}
\]
whence with $\f{M}_k := \f{V}_k ^* \f{U}_k$, 
\begin{equation}\label{eq:proto_cov}
		\begin{array}{rl}
			\mean{x_{\gra} x_{\gra} ^*}^{-1} =: C_{\gra} &= \( \inv{\f{M}_k} \f{K}_k \)^* C_{\gra-2} \( \inv{\f{M}_k} \f{K}_k \) \\
			C_2 &= 	\(\inv{\f{M}_k} \f{K}_k\)^* \(\inv{\f{M}_k} \f{K}_k \)\\
C_1 &= \f{K}_k.
\end{array}
\end{equation}

The work in \cite{rue11} now immediately follows as a special case of $r=1, k=0$, with $p_1(\lap) = \grk^2 + \lap$, and $\f{V}_0 = \f{U}_0 = W_0$, the \emph{Whitney map}, with $\f{K}_0$ being the symmetric FEM stiffness matrix and $\f{M}_0$ being the consistent mass-matrix.

Other discrete and convergent approximations may be obtained for other choices of $\f{V}_k, \f{U}_k$. Over polytopal meshes shape-constrained to bound the gradient of the Wachpress coordinates \cite{GradBoundsGBCFloaterGillette14}- a \emph{generalized barycentric coordinate} - in norm, with $k=0$, $\f{U}_0 = \f{V}_0$ may be set as the Generalized Whitney Interpolant, $W_0$, \cite{DiffFormsPolytopeChristiansen08} instantiated by the Wachpress coordinates to obtain a convergent approximation when $r=0$ in the deterministic case.

Here we show that another, particularly simple, approximation may be obtained for $k \in \{0, m\}$ that foregoes the recursive phrasing of (\ref{eq:polylap_recursive}), (\ref{eqn:polylap_recursive_weakform}) and allows a \emph{universal approximation scheme} for solving (\ref{eqn:man_spde}) in the mean-square sense agnostic to $\gra, \grk$ for many manifolds of arbitrary dimension. We also show that the covariances of simplicially averaged (smoothed) gaussians solving (\ref{eqn:man_spde}) shall also be approximated with the same approximation scheme. We shall further show that, with a little additional mesh-regularity, such an approximation scheme may be computed very quickly using existing solvers.

Specifically, we show:

\input{thm.main.2}

We close this section with an anticipatory remark.

\begin{remark}\label{rmk:weakform_interp}
	The weak-form posed in (\ref{eqn:protobase_weakform}) is considered well-defined even if $p_r (\lap)^{1/2} v$ is \emph{not well-defined}, if (\ref{eqn:protobase_weakform}) may be posed using the properties of $p_r(\lap), \f{V}_k, \f{U}_k$ in some form that is well-defined, e.g. (\ref{eq:proto_base}), (\ref{eqn:protobase_precision}). This notion is consistent with the Galerkin form being well-defined for piecewise linear functions, and (\ref{eqn:man_spde}) being well-defined for random-processes.
\end{remark}

%% file: thm.main.2.tex
\begin{thm}[Universal Approximation] \label{thm:main}
	Let $M_h$ be a well-centered $s$-relative-volume-concentrated \footnote{See \S\ref{sec:Approx}} mesh of $\man{M}$ for $s < 1/2$. Then there exists a $k$-dimensional subspace, $\c{S}_k \subset \reals{n}$ such that the two-parameter family of covariance matrices, $C^{\gra, \grk} = \Cov(x)$, from $\f{U} x$ solving (\ref{eqn:man_spde}) over $M_h$ are approximated with relative error, $\gre < 1$, with a single projection, $\tilde{\grP}_k$, onto $\c{S}_k$, with $k=O(s/\gre)$:
	\[  \norm{(I -  \tilde{\grP}_k) C^{\gra, \grk}} \leq \gre \norm{C^{\gra, \grk}} \; \; \forall \gra \in \nat, \grk \in \real.  \]
	Similarly, if $\bar{x}$ is the simplicial-mean over $M_h$ of $\f{U} x$ solving (\ref{eqn:man_spde}), with $\bar{C}^{\gra, \grk} = \Cov(\bar{x})$, then there exists a $k$-dimensional subspace $\bar{\c{S}}_k \subset \reals{n_m}$ such that $\bar{C}^{\gra, \grk}$ is approximated with relative error $\gre' < 1$ with a single projection, $\tilde{grP}' _k$ onto $\bar{\c{S}}_k$ with $k=O(s/\gre')$:
	\[  \norm{(I -  \tilde{\grP}'_k) \bar{C}^{\gra, \grk}} \leq \gre' \norm{\bar{C}^{\gra, \grk}} \; \; \forall \gra \in \nat, \grk \in \real.  \]
	Furthermore, an orthonormal basis, $V$, for $\c{S}_k$ can be deterministically computed in $O(n_1 k + n_0 k^2 + k^3)$, and $\bar{V}$ for $\bar{\c{S}}_k$ can be computed in $O(n_1 k + n_m k^2 + k^3)$, where $n_i$ is the number of $i$-faces (see \S\ref{sec:ectools}) in $M_h$. 
\end{thm}

%% file: S.ECTools.tex
\section{Tools from Discrete Exterior Calculus} \label{sec:ectools}
We assume familiarity with riemannian manifolds and exterior calculus. We point the reader to \cite{AMR88} for the continuous theory, \cite{whitney57geoint} for the theory of integration and the De Rham map which acts as a bridge to the discrete and will model the measurement-operator, to \cite{Hirani03, Desbrun03, Kalyanaraman16Hodge} for the theory of DEC, to \cite{AFWBook18} for the theory of FEEC, and to \cite{CraneDDG, CraneEC23} for a concise visual reconciliation of the continuous and the discrete. 

Here, we provide a contextually relevant summary of necessary topics, including the recent proof of convergence of the DEC by \cite{GP26}, and develop tools for use in \S\ref{sec:Approx} where we obtain a universal approximant. 

\subsection{$m$-plicial Complexes}
Let $\man{M}, \grm, M$ be as in Notation \ref{nota:basicManifold}. An $m$-simplex shall simply be called an \emph{$m$-plex}. An $m$-plex is said to be \emph{well-centered} if its circumcenter lies strictly within it. An $m$-plicial complex, $M$, is well-centered if all its constituent $k$-plices are, $\forall k \leq m$. $M$ is additionally \emph{shape-regular} if the ratio of circumradius to inradius is bounded across all constituent $k$-plices for all $k$. 

\begin{assume} \label{ass:goodComplex}
$M$ is a well-centered, shape-regular triangulation of $\man{M}$.
\end{assume}

\subsection{Chains, Cochains, Exterior Derivative and Co-Derivative}
The evaluation pairing of a $k$-plex, $\grs_k \in M$, with a $k$-form, $\grw_k \in \grL^k (M)$, is
\[ \< \grs_k, \grw_k \> = \int_{\grs_k} \grw_k = (\c{R}_k (\grw_k)) [\grs_k]. \]
The evaluation-map given above is the \emph{De Rham Map}, $\c{R}_k$ of order $k$. $\c{C}_k$ is the space of $k$-chains over $M$ - formal linear combinations of $k$-plices in $M$ - and $\c{C}^k$ it dual-space of $k$-cochains, both isomorphic to $\reals{n_k}$. $\partial_k: \c{C}_k \to \c{C}_{k-1}$ is the boundary-operator mapping an oriented $k$-plex to its oriented $(k-1)$-plicial boundary. As in the continuous case, its dual with respect to the evaluation-pairing, $\< \grs_k, f \>$, is the DEC exterior derivative \cite{AMR88}, $\dd_k: \c{C}^k \to \c{C}^{k+1}$: $\< \partial_{k+1} \grs_{k+1}, y_k \> = \< \grs_k, \dd_k y_k \>$ for $y_k \in \c{C}^k$.

As in the continuous case, $\{ \c{C}^k, 0 \leq k \leq m \}$ forms a discrete \emph{De Rham complex} with $\dd: \c{C}^k \to \c{C}^{k+1}$. $\c{R}$ shall be used to denote the De Rham map onto this complex where necessary, with context implying $k$.

The De Rham complex forms a \emph{graded}\footnote{The 'grade' being the form-degree, $k$} vector space and is equipped with an inner-product, $(u, v): \c{C}^k \times \c{C}^k \to \real$. The \emph{codifferential}, $\dad$, is the adjoint of $\dd$:
\[ (u, dv)_{k+1} = (\dad u, v)_k. \]

\subsection{Interpolation Maps}
The notion of an interpolant is fundamental. It is (re)defined here in a contextually convenient manner.
\begin{defn}[Interpolation Map/Interpolant] \label{defn:interpolationMap}
	A map, $\c{I}_k: \c{C}^k \to  \grL^k M$, is said to be an interpolation map if $\Dim (\c{I}_k (\c{C}^k)) = \Dim (\c{C}^k)$. It is $t$-continuous if its range is a subset of $C^t \grL^k M$.
\end{defn}

\begin{remark}\label{rmk:interpolationMap}
	Under Definition \ref{defn:interpolationMap}, we have that a map $\c{C}^k \to C^0 \grL^{k}M_{q+k}$ may be an interpolant for $q \geq 0$. Specifically, $\c{R}^*$ is an interpolant in this definition. 
\end{remark}

We say that $\c{I}_k$ \emph{respects the De Rham map} if $\c{R}_k \c{I}_k = I_{n_k}$. In the remainder of this work, we assume
\begin{assume} \label{ass:respectDeRham}
	For $0 \leq k \leq m$, the interpolation maps, $\c{I}_k: \reals{n_k} \to C \grL^k (X)$ respect the De Rham map, and $\{ \c{I}_k \c{R}_k \}$ forms a bounded cochain projection of $\grL^k \man{M}$ \cite{AFW06}. 
\end{assume}

The Whitney and Generalized Whitney interpolants \cite{DiffFormsPolytopeChristiansen08} satisfy Assumption \ref{ass:respectDeRham}, as do Generalized Barycentric Coordinates \cite{Floater15}, as a special case.

\subsection{Circumcentric Dual}
The polytopal \emph{circumcentric dual}, $M'$, of $M$, is a collection of all $\astar \grs, \grs \in M$, where $\astar$ associates each $k$-plex, $\grs_k$, in $M$ with an orthogonal $(m-k)$-cell (dual-cell) \footnote{uniquely defined by the circumcenters of all $t$-plices in $M$ containing $\grs_k$ for $t\geq k$}, $\grp_{m-k} = \astar \grs_k$. The circumcentric cell-dual operator, $\astar$, is an automorphism modulo orientation, reversal of which is denoted by $-1$: 
\[ \astar \astar \grs_k = (-1)^{k(m-k)} \grs_k. \]
$M'$ is a polytopal complex in general, and we have \footnote{modulo orientations of simplices, which agree with $\star \star$ and hence leave the space of $k$-cochains invariant} $(M')' \equiv M$.

\begin{notation} \label{nota:genCell}
 Let $\bar{\c{C}}_k, \bar{\c{C}}^k$ denote the space of $k$-chains and $k$-cochains over $M'$, and $\bar{\c{R}}_k, \bar{\c{I}}_k$ denote the De Rham map onto $M'$ and the interpolation map from $\reals{\bar{n}_k}$ respectively. 
Let $X \in \{ M, M' \}$ denote the generic cell-complex with $X'$ well-defined. Let $\grt_k \in X$  denote the generic $k$-cell in $X$. Furthermore, $\c{R}_k, \c{I}_k, \c{C}_k, \c{C}^k$ shall all denote the relevant objects over $X$, with $\c{R}_k ', \c{I}_k ', \c{C}' _k, \c{C}^{'k}$ denoting objects in $X'$, as shall be made evident by the context.  
\end{notation}

 Similar to $\{\c{C}^k, 0 \leq k \leq m\}$, a discrete De Rham complex is formed by $\{ \bar{\c{C}}^k, 0 \leq k \leq m \}$ with $\bar{\dd}: \bar{\c{C}}^k \to \bar{\c{C}}^{k+1}$. 

\subsection{DEC Hodge Star}
We recall that the continuous hodge-star, $\star: \grL^k T_z \to \grL^{m-k} T_z$, is a pointwise isometry for $z \in \man{M}$, across $\grm$-orthogonal-complements of $\grL^k T_z$ leaving coefficients in the Pl\"ucker coordinates invariant. 

The orthogonal relationship between $M, M'$ serves to mimic this in the discrete setting and in the average sense: where $(\perp_{\grm} \grw_k, \star \grf_k(z)) = (\grw_k, \grf_k(z))$ for all $k$-vectors, $\grw_k \in \grL^k T_z$ and their (positively oriented) $\grm_z$-orthonormal complement $(m-k)$-vectors, $\perp_{\grm} \grw_k \in \grL^{m-k} T_z$, one has in the discrete case\footnote{where $M$ may be thought of as $\man{M}$ countably ``flattened'' by ``impressing'' a countable collection of its tangent-spaces, $T_{z_i}$, into it, and recovered in the limit where this flattening becomes uncountable.} for $\grt_k \in X$ (as in Notation \ref{nota:genCell}): 
\begin{equation} \label{eqn:dstar}
	\frac{f_{ki} = \<\grt_{k, i}, \grf_k\>}{\abs{\grt_{k, i}}} = \frac{(\dstar f_k)_i \stackrel{(a)}{\asymp} \<\astar \grt_{k,i}, \star \grf_k\>}{\abs{\astar \grt_{k,i}}},
\end{equation} 
leading to the definition of $\dstar$: 
\begin{defn}[Discrete Hodge Star]
	Using Notation \ref{nota:genCell}, the discrete hodge star, $\dstar: \c{C}^k \to \c{C}^{'(m-k)}$, is defined on the $k$-cochains of $X$ as the $n_k \times n_k$ diagonal matrix with
	\[ \dstar_k (i, i) := \frac{\abs{\astar \grt_{k, i}}}{\abs{\grt_{k, i}}}. \]
\end{defn}
The \emph{shape-regularity} constraint on $M$ ensures that $\dstar$ is well-conditioned.
The asymptotic equality, $(a)$, in (\ref{eqn:dstar}) is not a prerequisite in the definition of $\dstar$ \cite{Hirani03, Desbrun03}. Its presentation here serves to define the \emph{discrete-hodge-consistency} of an interpolant:

\subsection{Discrete-Hodge-Consistency}
\begin{defn}[Discrete-Hodge-Consistency]
	For $X_h \in \{ M_h, M' _h \}$, a $k$-form, $\grf_k \in C \grL^k (X_h)$ is said to be discrete-hodge-consistent if $(a)$ in (\ref{eqn:dstar}) is an exact equality modulo $O(h)$:
	\begin{equation} \label{eqn:discreteHodgeConsistency}
		\c{R}' _{m-k} \star \grf_k = (\dstar + O(h)) \c{R}_k \grf_k \heq \dstar \c{R}_k \grf_k.
	\end{equation}
	An interpolant-map, $\c{I}_k : \reals{n_k} \to C \grL^k (M_h)$, is said to be discrete-hodge-consistent if it has a discrete-hodge-consistent basis. 
\end{defn}

Discrete-hodge-consistency qualifies the mimesis of the continuous $\star$, where DEC readily reconciles with FEEC. For instance, as shown in \cite{GP26, TS16DEC},

\begin{remark} \label{rmk:DiscreteHodgeConsistency}
Whitney interpolants over well-centered, shape-regular simplicial complexes and Generalized Whitney interpolants over their circumcentric duals are both discrete-hodge-consistent.
\end{remark} 

This follows from an observation about forms over cells containing both $\grs_{k, i}, \astar \grs_{k, i}$ (made in \S$5.2$ in \cite{GP26}). We summarize this briefly and refer the reader to the appendix for more detail.

\begin{defn}[Flange]
	The geodesic convex hull in $M$ of $\grs_{k, i} \cup \astar \grs_{k, i}$ is called the flange, $\flange{k, i}$, of $\grs_{k, i}$ (or the dual-flange of $\astar \grs_{k, i}$). In general, it is called a $k$-flange (or an $(m-k)$-dual-flange). $F_k:= \{ \flange{k, i} \}$ is the $k$-flange-complex (or $(m-k)$-dual-flange-complex) of $M$.
\end{defn}
\begin{remark}\label{rmk:flange}
	A $k$-flange is $m$-dimensional for all $k$. It is called a $k$-flange since it is uniquely defined in $M$ by $\grs_{k, i}$ (or $\grp_{m-k, i}$).
\end{remark}
The $m$-cells in $M, M', F_0, \cdots, F_m$, all arise from different unions of $m$-plices in the circumcentric subdivision of $M$ \footnote{Similar to a barycentric subdivision \cite{AMR88} but with circumcenters.}. 
A flange is called the \emph{diamond-cell} in \cite{GP26}. All constant $k$-forms on a flange clearly have the property that $(a)$ in (\ref{eqn:dstar}) is an \emph{exact} equality for all $h$, and all forms they approximate with relative error $O(h)$ are discrete-hodge-consistent. Since the flange is convex it is star-shaped and we have, by Bramble-Hilbert lemma applied to $k$-forms [cite], that constant forms, $\bar{\grw}_{\flange{}}$, approximate smooth forms, $\grw$, with $O(h)$ relative error. So the smooth form, $\grw$, is discrete-hodge-consistent over well-centered, shape-regular meshes;  $(a)$ in (\ref{eqn:dstar}) holds and $\dstar$ convergently mimics $\star$ \footnote{One is still required to show that certain non-smooth forms are discrete-hodge-consistent to use linear interpolants to solve a Hodge-Poisson problem. This can be found in the appendix in this version or later versions, as well as in \cite{GP26}.}
One may show that the coderivative is preserved up to $O(h)$:
\input{lma.approx_dstar.claim}

This is instrumental in defining the operator central to a Hodge-Poisson problem: the DEC Hodge-Laplacian. 

\subsection{DEC Hodge-Laplacian}

\begin{defn}[DEC Hodge-Laplacian] \label{defn:hodge_laplacian}
	The DEC hodge-laplacian, $\dlap_k$, of order $k$ over $M$ is
	\[ \dlap_k = \dd_{k-1} \bar{\dstar}_{m-k+1} \bar{\dd}_{m-k} \dstar_k + \bar{\dstar}_{m-k} \bar{\dd}_{m-k-1} \dstar_{k+1} \dd_k = \dd \bar{\dstar} \bar{\dd} \dstar + \bar{\dstar} \bar{\dd} \dstar \dd \]
	and its equivalent, $\bar{\dlap}_k$, over $M'$ is
	\[ \bar{\dlap}_k = \bar{\dd} \dstar \dd \bar{\dstar} + \dstar \dd \bar{\dstar} \bar{\dd}.\] 
	The form-order shall be contextually implied when $\bar{\dstar}, \bar{\dd}, \dstar, \dd$ are used. 
\end{defn}

\begin{notation}\label{nota:hodgeLaplacian}
	Extending Notation \ref{nota:genCell}, when $X \in \{M, M'\}$ denotes the generic complex, $\dlap = \dd \dstar \dd \dstar + \dstar \dd \dstar \dd$ shall denote the DEC hodge-laplacian on $\c{C}^k(X)$, with $\dlap'$ denoting the DEC hodge-laplacian on $\c{C}^{'k}$. 
\end{notation}


\subsection{Distributional Interpretation of Discrete Exterior Calculus}
DEC is posed here using the theory of distributions, as hinted by Remark \ref{rmk:weakform_interp}, to obviate the closeness to the approach of \cite{rue11, rue22}: with Notation \ref{nota:basicManifold} one sees that the De Rham map, $\c{R}_k$, is linear and
\[ \c{R}_k (\grf) [i] = (e_i, \c{R}_k \grf) =  \< \grs_{ki}, \grf \>,    \]
whence
\begin{equation}\label{eqn:distributionalDR}  
(\c{R}_k ^* e_i, \grf) = \< \grs_{ki}, \grf \>.  
 \end{equation}
Clearly, $\c{R}_k ^* e_i$ is the distribution that is $1$ on $\grs_{ki}$ and $0$ everywhere else. As a result, one may consider $\c{R}_k ^* : \reals{n_k} \to \reals{\man{M}}$ to be an interpolation map, albeit a trivial one, and $Range(\c{R}_k ^*)$ to be the test-space. We may thus restate (\ref{eq:polylap_recursive}) as 
\begin{equation}\label{eqn:polylap_recursive_DEC}
	(\c{R}_k ^* z , p_r(\lap) \c{I}_k x_t) = (\c{R}_k ^* z, \c{I}_k x_{t-2}) \Rightarrow \c{R}_k p_r(\lap) \c{I}_k x_t = x_{t-2}.
\end{equation}
\begin{remark} \label{rmk:distributionalDEC}
This distributional phrasing uses the "observable" space of functions/forms - over the relevant $k$-plices - to test the weak formulation as opposed to interpolating a test-space. Since DEC is shown to have $O(h)$-convergence, little is lost in this process.
\end{remark}
The construction of DEC now provides the following.

\input{lma.prd}

\begin{remark} \label{rmk:vemLike}
As $h \to 0$, $\c{R}_k p_r (\lap_k) \c{I}_k = p_r(\dlap_k) + O(h) \to p_r(\dlap_k)$ as in VEM, which avoids the explicit computation of $\f{V}^* _0 \lap_0 \f{U}_0$ precisely by providing a stable approximation that is $O(h)$ away from the exact mass-normalized stiffness matrix, and converges to this normalized stiffness matrix at the same rate that the normalized stiffness matrix converges to the Laplace-Beltrami operator. 
\end{remark}
With the convergent, discrete analogue for $p_r (\lap_k)$ developed, we establish a convergent solution to equations such as $(\ref{eqn:man_spde})$:  
Suppose 
\[ p_r (\lap_k) u = \grf \]
is discretized as 
\[ (p_r (\dlap_k) + O(h)) x = \c{R}_k \grf = f. \] 
\input{prop.prd_solns_converge}
We have shown that not much is lost by omitting the $O(h)$ term from $p_r (\dlap_k) + O(h)$. So we may henceforth assume\footnote{A self-contained convergence-result for $\c{R} \lap_k \c{I}$ may be made available here in later versions, from which Proposition \ref{prop:prd_soln_convergence} gives the result here by the triangle inequality}  
that $\c{R} p_r (\lap_k) \c{I}$ is discretized as $p_r (\dlap_k)$, following the results of \cite{GP26}, which the following proposition does. 
\input{prop.prd_soln_convergence}

\input{lma.dual_dlap}

\input{corl.mto0_dlap}

We now state and prove the main result of this section.
\input{thm.decgmrf}

We observe that the cell-wise means\footnote{This notion is rigorously defined via a homeomorphism from $M \to \c{M}$ that maps $\grs_k$ to its geodesic equivalent, and pulling the De Rham map back to $M$} of $u$ over $X_m \in \{ M, M' \}$ solve the $p_r (\dlap' _0)$ over $X' _0$; such smoothed approximations are convergent. We shall see in the following section that, as a result, one may have an eigendecomposition for smoothed approximations to GMRFs over such cells.

%% file: lma.approx_dstar.claim.tex
\begin{lemma}\label{lma:approx_dstar}
	Let $\c{I}_k$ be an interpolant satisfying Assumption \ref{ass:respectDeRham}. Then 
	\[ \c{R}_k \grd \c{I}_{k+1} = \c{R} _{k} \star d \star \c{I}_{k+1} \heq \dstar \dd \dstar. \]
\end{lemma}

%% file: lma.prd.tex
\begin{lemma}\label{lma:prd}
	Let $\man{M}, M, \lap_k, p_r(\lap_k)$ be from Notation \ref{nota:basicManifold} and let $M', X, \dlap_k, \c{R}_k, \c{I}_k$ be from Notations \ref{nota:genCell}. Then
	\[ \c{R}_k p_r(\lap_k) \c{I}_k = p_r (\dlap_k) + O(h), \;\; \dlap_k = \dstar \dd \dstar \dd + \dd \dstar \dd \dstar.  \]
\end{lemma}

\begin{proof}
	From Lemma \ref{lma:approx_dstar} we have
	\[ 
	\begin{array}{rl}
		\c{R}_k \star d \star d &= \( \dstar + O(h) \) \c{R}' _{m-k} d \star d \\
		&= \( \dstar +O(h) \) \dd \c{R}' _{m-k-1} \star d  \\
		&= \( \dstar + O(h) \) \dd \( \dstar + O(h) \) \c{R}_{k+1} d \\
		&= \( \dstar + O(h) \) \dd \( \dstar + O(h) \) \dd \c{R}_k \\
		&= \( \dstar \dd \dstar \dd + O(h) \) \c{R}_k \\
	\end{array}
	\]
	Similarly,
	\[	\c{R}_k d \star d \star = \( \dd \dstar \dd \dstar + O(h)\) \c{R}_k, \]
	and summation gives
	\begin{equation} \label{eqn:contToDiscLap}
		\c{R}_k \lap_k = \( \dlap_k + O(h) \) \c{I}_k.
	\end{equation}
	So for the monomials $\c{R}_k \lap_k ^t = \( \dlap_k + O(h) \)^t \c{R}_k \forall t \in \nat$. Linearity now provides 
	\[ \c{R}_k p_r (\lap_k) = p_r (\dlap_k + O(h)) \c{R}_k.\] 
	Since $\c{I}_k$ respects the De Rham map, $\c{R}_k \c{I}_k = I_{n_k}$ and
	\[ \c{R}_k p_r(\lap_k) \c{I}_k = p_r (\dlap_k + O(h)) \c{R}_k \c{I}_k = p_r(\dlap_k + O(h)) = p_r(\dlap_k) + O(h)  \]
	as claimed.
\end{proof}

%% file: prop.prd_solns_converge.tex
\begin{prop} \label{prop:prd_solns_converge}
	Let $p_r (\dlap_k) \tilde{x} = f$ and $\( p_r(\dlap_k) + H \) x = f$ for $H=O(h), |h| \ll 1$. Then $\norm{x - \tilde{x}} = O(h) \norm{f}.$ 
\end{prop}
\begin{proof}
	For $B = p_r(\dlap_k)$, one has $(I + H \inv{B})B x = f$, with $B$ being invertible since $p_r(\lap)$ is strictly positive by assumption, whence, over a shape-regular mesh, discrete poincare inequalities grant $p_r(\dlap) \Gt 0$. Now,
	\[ \begin{array}{rl} 
		x &= \inv{B} \( I + H \inv{B} \)^{-1} f \\
		&= \inv{B}\(f - H \inv{B}f \) + O(h^2)f  \quad \quad \text{Taylor expansion, } H = O(h)\\
		&= \tilde{x} - \inv{B} H \inv{B} f + \inv{B} O(h^2) f.\\
	\end{array} \]
	So, by Cauchy-Schwarz,
	\[ \norm{x - \tilde{x}}^2 = \norm{\inv{B} H \inv{B} f}^2 + O(h^4) \leq \norm{\inv{B} H \inv{B}}^2 \norm{f}^2 + O(h^4) = O(h^2) \norm{f}^2,\]
	the square root of both sides then proving the claim.
\end{proof}

%% file: prop.prd_soln_convergence.tex
\begin{prop}\label{prop:prd_soln_convergence}
Suppose $k \in \{ 0, m \}$ and the deterministic problem, 
	\[ \( \grk^2 + \lap_k \) u = \grf \]
	is discretized over $X_h \in \{ M_h, M'_h \}$ as 
	\[ \( \grk^2 I_{n_k} + \dlap_k \) x_h = f,  \]
Let the approximation error, $e_h$, be defined as
	\[ e_h = u - \c{I}_k x_h. \]
Then there is a unique $x_h$ that can be computed for which $\norm{e_h } _{M} = O(h) \norm{\grf}_{M}$. 
Furthermore, if $\grf = \c{W}$, a white noise process over $\man{M}$, and $u$ is approximated in the mean-square sense, there is a unique covariance,
	\[ \norm{\Cov(u) - \Cov(\c{I}_k x_h)} = O(h^2). \]
\end{prop}
\begin{proof}
	The uniqueness of $x_h$ comes from the least positive eigenvalue, $\grl_{1}(\dlap_k)$, of $\dlap_k$ being bounded away from 0. This is given by the discrete poincare inequalities of \cite{TS16DEC} for $k \in \{0, m \}$, by \cite{Kaibo26} for $k \leq m \in \{1, 2 \}$ and by \cite{GP26} for all $m, k$. 
	The error, $e_h$, being controlled in norm by $\grf$ as $O(h)$ in the deterministic case follows directly from recent results from \cite{GP26}, Theorem 5.2. 

For the stochastic case when $\grf$ is a random process over $\man{M}$, we have that
	\[ \norm{\Cov(u) - \Cov(u_h) } \leq \norm{\Cov(u - u_h)} = \norm{\Cov(e_h)} \leq O(h^2) \norm{ \Cov(\grf) }. \]
	 When $\grf$ is the white noise in particular, $\Cov(\grf) (x,y) = \Cov(\c{W}) (x, y) = \delta_{xy}$, giving the result. 
\end{proof}

%% file: lma.dual_dlap.tex
\begin{lemma}\label{lma:dual_dlap}
	When $k(m-k) \in 2 \mathbb{N}$, $\dstar x$ solves 
	\[  p_r (\dlap_{m-k}) \dstar x = \dstar f \]
	 for $\grr \in \rats$ if and only if $x$ solves
	\[ p_r (\dlap_k ')  x = f. \]
\end{lemma}

\begin{proof}
	Clearly, $\dstar: \c{C}^s (X) \to \c{C}^{m-s} (X')$ is invertible, so
	\[  p_r(\dlap_k) x = f \iff \dstar  p_r(\dlap_k)  x = \dstar f \]
	Since $\dstar \dstar = I$ when $k(m-k) \in 2 \nat$, one has $x = \dstar \dstar x$ and
	\[ \dstar p_r (\dlap_k) x = \dstar f \iff \dstar  p_r(\dlap_k) \dstar (\dstar x) = \dstar f. \]
	Clearly, $\dstar \dlap_k ^t \dstar = (\dlap' _{m-k}) ^t \forall t \in \nat$:
	\[ \dstar (\dstar \dd \dstar \dd + \dd \dstar \dd \dstar) \dstar = \dd \dstar \dd \dstar + \dstar \dd \dstar \dd = \dlap' _{m-k}: \c{C}^{m-k} (X') \to \c{C}^{m-k} (X').  \]
	Since constants, $a$, are invariant under $\dstar a \dstar$ one has that 
	\[ \dstar p_r (\dlap_k) \dstar = p_r (\dlap' _{m-k}), \]
	proving the claim.
\end{proof}
Observing that when $k=m$, one has $\dstar(i, i) = 1/\abs{\grs_{i}}$ immediately provides the following useful corollary


%% file: corl.mto0_dlap.tex
\begin{corl}\label{corl:mto0_dlap}
	Let $\grs_i$ be the $m$-cell indexed by $i$ in $X \in \{ M, M' \}$.
	A discrete $m$-form $y \in \grL^m (X)$ solves 
	\[ p_r (\dlap_m) y = f, f \in \reals{n_m} \]
	if and only if $\bar{y}; \bar{y}_i := y_i/\abs{\grs_i}$ solves 
	\[ p_r (\dlap_0 ') \bar{y} = \bar{f}. \;\; \bar{f}_i := f_i /\abs{\grs_i}. \]
\end{corl}

%% file: thm.decgmrf.tex
\begin{thm}\label{thm:decgmrf}
	Let $x: (\grW, \c{F}, \Prob) \times M_0 \to \real$ be the values over the nodes of $M$ of $u$ solving (\ref{eqn:man_spde}) over $\man{M}$, with $(\grW, \c{F}, \Prob)$ being a suitable probability space. Let $P:= \grk^2 I_n + \dlap_0$. Then $x$ is a Gauss Markov Random Field with covariance, $C(\gra, \grk)$, given by 
	\begin{equation}\label{eqn:theorem2_1}
		C(\gra, \grk) = \((P^{\floor{\gra/2}})^{\top} (P^{\ceil{\gra/2}})\)^{-1} + O(h^2)
	\end{equation}
	Let $y \in (\grW, \c{F}, \Prob') \times M' _0 \to \real$ be the $m$-plexwise means of $u$ solving (\ref{eqn:man_spde}) over $\man{M}$. Let $\bar{P}:= \grk^2 I_{n'} + \bar{\dlap}_0$. Then $y$ is a Gauss Markov Random Field with covariance, $\bar{C}(\gra, \grk)$, given by 
	\begin{equation}\label{eqn:theorem2_2}
		\bar{C}(\gra, \grk) = \( (\bar{P}^{\floor{\gra/2}})^{\top} (\bar{P})^{\ceil{\gra/2}}. \)^{-1}  + O(h^2)
	\end{equation}
\end{thm}

\begin{proof}
	We have from the definition of the De Rham map, that $x = \c{R}_0 u$. We have from \ref{prop:prd_soln_convergence} that $v_{\gra}$ is an $O(h)$-convergent solution to the (distributional) weak-form, (\ref{eqn:polylap_recursive_weakform}), with $\f{U}_0 = \f{V} _0$, 
	\[ (\c{R}_0 ^* z, \grk^2 + \lap_0 \c{I}_0 v_{\gra}) = (\c{R}_0 ^* z, \c{I}_0 v_{\gra-2}), \;\; \gra \geq 2. \]
	Since $\grk^2 + \lap_0$ is a strictly positive polynomial in $\lap_0$, we have by Lemma \ref{lma:prd} from setting $p_r(\lap_0) = \grk^2 + \lap_0$ and the definition of $P$ that
	\[ Pv_{\gra}= v_{\gra-2} \Rightarrow v_{\gra} = \inv{P} v_{\gra-2}, \;\; \gra \geq 2. \]
	For $\gra \in 2 \nat$ we have from the base case in (\ref{eqn:polylap_recursive_weakform}) that $v_0 = \c{R}_0 \c{W}$, which is distributed as $\c{N}(0, 1)$ since $\c{W}$ is the white noise, and:
	\[ \tilde{x} = v_{\gra} = (\inv{P})^{\gra/2} v_0, \;\; v_0 \sim \c{N}(0, 1),  \] 
	whence
	\[  \Cov(\tilde{x}) = \mean{\tilde{x} \tilde{x}^{\top}} = \( (P^{\gra/2})^{\top} P^{\gra/2} \)^{-1} =\( (P^{\floor{\gra/2}})^{\top} (P)^{\ceil{\gra/2}}.\)^{-1}  \]
	For $\gra \in 2 \nat + 1$ we have from (\ref{eqn:protobase_precision}) with $\f{V}_0 ^* = \c{R}_0$ that $\f{K} = P$ whence
	\[ \Cov(\tilde{x}) = \mean{\tilde{x} \tilde{x}^{\top}} = \( (P^{(\gra-1)/2})^{\top} P (P)^{(\gra -1)/2} \)^{-1} = \( (P^{\floor{\gra/2}})^{\top} (P)^{\ceil{\gra/2}}\)^{-1}. \]
	Since $\norm{\tilde{x} - x} = O(h)$, we have that 
	\[ \Cov(x) - \Cov(\tilde{x}) = O(h^2),  \]
	whence
	\[ C(\gra, \grk) =  \( (P^{\floor{\gra/2}})^{\top} (P^{\ceil{\gra/2}}) \)^{-1} + O(h^2) \]
	as posited in (\ref{eqn:theorem2_1}).
	Now Corollary (\ref{corl:mto0_dlap}) gives (\ref{eqn:theorem2_2}) arguing as above.
\end{proof}

%% file: S.Approximation.tex
\section{Universal Approximation Scheme} \label{sec:Approx}
Given the covariance, $C(\gra, \grk)$, of the random process defined by (\ref{eqn:man_spde}) and convergently approximated over a discretisation, $M$, of $\man{M}$, as in Theorem \ref{thm:decgmrf} we turn to the problem of approximating $C(\gra, \grk)$.  

Let $U(\gra, \grk), \grS(\gra, \grk)$ be the eigenvectors and eigenvalues of $C(\gra, \grk)$; then we recall that $\grl_i = O(i^{-\gra})$ \cite{kelner09higher}[cite manifolds instead] and that this decay allows one to obtain, by the Eckart-Young theorem, a $k$-rank approximation of $C(\gra, \grk)$ with an error that is $O( (k+1)^{-\gra})$. This observation continues beyond the case of finite dimensions into countable dimensions as was made foundationally by [continuous guy 1958] and stated for the case of gaussian random variables by Kosambi, Karhunen and Loeve [cite KKL]. Here we show that, by such a truncation, we may efficiently obtain a relative error approximation to a gaussian process over $M$, \emph{agnostic} to $\gra, \grk$.

\subsection{Parameter Agnosticity/Universality}  
The fourier basis clearly diagonalizes $(\grk^2 + \lap)^{\gra/2}$ for all $\grk, \gra$, with the spectrum absorbing the dependence on $\grk, \gra$. This property is preserved in the discrete precision matrices, $P, \bar{P}$, in Theorem \ref{thm:decgmrf}:
\begin{corl} \label{corl:param_agnostic}
	Let $\dlap_0 = U \grS V^{\top}, \dlap' _0 = \bar{U} \bar{\grS} \bar{V}^{\top}$ be the SVDs of $\dlap_0, \dlap' _0$ respectively. Then
	\[ P = U (\grk^2 + \grS) V^{\top}, \;\; \bar{P} = \bar{U} (\grk^2 + \bar{\grS}) \bar{V}^{\top} \]
	and 
	\begin{equation}\label{eqn:param_agnostic} 
		\begin{array}{rl}
			\inv{C}(\gra, \grk) &= V \( \grk^2 + \grS)^{\gra} V^{\top}, \\
			\inv{\bar{C}}(\gra, \grk) &= \bar{V} (\grk^2 + \bar{\grS} )^{\gra} \bar{V}^{\top}.
		\end{array}
	\end{equation}
	Furthermore, if $\grP_k, \bar{\grP}_k$ are orthogonal projections onto any $k$-dimensional eigenspace of $C(\gra, \grk), \bar{C}(\gra, \grk)$ respectively, then $\grP_k, \bar{\grP}_k$ are invariant to $\gra, \grk$.
\end{corl}

We call this invariance of eigenprojections to $\gra, \grk$ as \emph{parameter agnosticity}. 

\begin{remark} \label{rmk:likeContinuous}
	The implication of (\ref{eqn:param_agnostic}) is that, for the specific case of $\gra \in \nat$,the discretization of an operator-exponent of the continuous linear operator, $(\grk^2 + \lap)^{\gra/2}$, is the operator-exponent of its discretization; operator-exponentiation commutes with discretization. The recursive phrasing of \cite{rue11} and (\ref{eq:polylap_recursive}) may thus be abandoned altogether, as mentioned in \S\ref{sec:protoscheme}. 
\end{remark}

The projections in Corollary \ref{corl:param_agnostic} depend only on the quality of the triangulation, $M$, and the geometry of $\man{M}$. However, the monotonicity in the spectral decay of the covariances with respect to $\gra$ allows one to obtain low-rank projections of the covariance matrices with \emph{uniformly bounded} relative error across all $\grk, \gra$, much like KKL in the continuous case.

\subsection{Approximation} 
We begin with an assumption that ensures that the DEC hodge-laplacian over the polytopal complex of a triangulation of $\man{M}$ has a highly delocalized or \emph{incoherent} bottom $k$-eigenspace. \footnote{This can be lemmatized by assuming certain geometric properties of $\man{M}$ (like non-negative Ricci curvature), or by equivalently assuming that the 1-skeleton of $M' _h$ has good expansion properties: obeying a poincare inequality, having uniform conductance, or a cheeger-constant or fiedler value bounded away from 0. This is reserved for later versions of this preprint in the interest of sharing the broader argument concisely first.} 

\input{assume.laplacian_incoherence}

A central tool in approximating the eigenspace of $\bar{\dlap} = \dstar' \bar{L}$ using that of $\bar{L}:= \dd \bar{\dstar} \bar{\dd}$ is the bound of Li \cite{liRelPert94} controlling the perturbation to singular spaces caused by a multiplicative perturbation. We capture this in a contextually relevant form:
\input{lma.li_bound}

\begin{remark}\label{rmk:no0_nolap}
	$L _k = \bar{\dd}_{m-k-1} \dstar_{k+1} \dd_k$ (resp. $\bar{L}_k$) are all symmetric positive semidefinite, but only $L_0$ (resp. $L' _0$) may be assumed to be graph-laplacians across all discrete topologies of $M$ (resp. $M'$).  
\end{remark}

\begin{notation} \label{nota:graph_laplacian}
	Let $L, \bar{L}$ be the graph-laplacians on the vertices of $M, M'$ respectively, with 
	\[ \dlap_0 = \bar{\dstar} L,\;\; \bar{\dlap}_0 = \dstar \bar{L},\quad \dlap_m = \bar{L} \dstar_m,\;\; \bar{\dlap}_m = \bar{\dstar_m} L.  \]
\end{notation}

The main approximation result may now be proven.

\input{thm.eigspace_approx}

\begin{remark} \label{rmk:rand-v-det}
	Sampling with $p_i \propto |\grs_i|$ and arguing similarly yields a sketching result with probabilistic guarantees requiring a higher dimensional projection. Instead we leave further approximations of $V_k$ to efficient eigencomputations with graph-laplacians, which may include sketching - via graph-sparsification for instance.
\end{remark}

\subsection{Algorithm/Scheme}

Theorem \ref{thm:eigspace_approx} results in an algorithm to compute an approximation to $C(\gra, \grk)$ that guarantees a relative error of $\gre < 1$ and, therefore, an algorithm to compute a low-dimensional representation of measurements taken of a GMRF, including that of simplicial means.

\subsubsection{Algorithm}

\begin{minimalgo}{$\mathtt{UniversalBasis}$ - Compute $\epsilon$-Approximating Basis}
	\textbf{Assumes:} $\mathtt{LaplacianEig}$ is a Laplacian Eigensolver \\
	\textbf{Input:} \begin{enumerate}
		\item[\quad] $(X_0, X_1, \dstar_1), \; X \in \{ M, M' \}$. 
		\item[\quad] Relative-error tolerance, $\gre > \grr_M s$.
	\end{enumerate}
	\textbf{Returns:}  $\tilde{V}_k \in \reals{n_0,k}$, $\tilde{V}_k ^{\top} \tilde{V}_k = I_k$, and $\tilde{\grP}_k = \tilde{V}_k \tilde{V}_k ^{\top}$ is a \emph{Universal Approximator} for covariance matrices, $C_{\gra, \grk}$, over $X_0$.
  \begin{enumerate}[label=\arabic*., leftmargin=1.5em, itemsep=2pt, topsep=8pt]
	\item $k \from \displaystyle\frac{\gre}{\grr_M s}.$ \comment{Rank of approximation}
	\item $L[i, j] \from (-\dstar_1) (l,l)$ such that $\grt_{1,l} = (\grt_{0, i}, \grt_{0, j})$ \comment{Initialize Laplacian to be adjacency}
	\item $\Diagonal(L) \from -L \mathbf{1}_{n_0}$ \comment{faster for $(d+1)$-regular $M'$} 
	\item $\tilde{V}_k \from  \mathtt{LaplacianEig}(L, k)$ \comment{May be implicitly computed in certain applications}
	\item Return $\tilde{V}_k$
  \end{enumerate}
\end{minimalgo}

\subsection{Runtime of $\mathtt{LaplacianEig}$}
The runtime of $\mathtt{UniversalBasis}$ depends on the runtime of $\mathtt{LaplacianEig}$. 
Many routines in numerical linear algebra address precisely this problem of computing the bottom $k$-SVD/EVD of a hermitian matrix. LOBPCG\cite{lobpcg1, lobpcg2} is often considered a good deterministic choice, requiring $O(n_1 k + n_0 k^2 )$ floating point operations (flops) per iteration \footnote{This is fast despite hiding $O(k^3)$ operations per iteration in the big $O$ notation}, with the total number of iterations being $O(1/\bar{\gre})$ to guarantee a relative error, $\bar{\gre}$, giving a total running time of $O((n_1 k + n_0 k^2 )/\bar{\gre})$. This has been improved to $O\( (n_1 k +  n_0 k^2)/\sqrt{\bar{\gre}} \)$ for the deterministic case in \cite{Lazysvd16}. 

Randomized algorithms offer more room for tradeoff between runtime and accuracy from their greater margins of failure compared to the deterministic case. A randomized SVD for this problem, allowing an additional rank-budget of $k$, provides a rank $2k$ approximation that contains the best rank $k$ approximation to $L$ with a high probability \cite{RST10, HMT11} and takes $O(n_1 k + n_0 k^2  + k^3)$ flops to obtain: $q n_1 (2k)$ from mutiplying a random, $(n_0, 2k)$, gaussian (or DFT) $q$ times\footnote{this initial multiplication is a common feature in fast SVD algorithms, the running time of [LazySVD] reflecting it by the $O(kn_1 + k^2 n_0)$ factor.}   with $L$, $n_0 (4k)^2$ to orthonormalize the result to get a basis, $Q$; $8k^3$ to find the deterministic SVD of $(Q^{\top} L) Q$. As the mesh-size gets large, this is recommended since $k$ is effectively fixed for this problem due to the nature of spectral decay in $L$. 

In this version, we shall take $O(n_1 k + n_0 k^2)/\sqrt{\gre}$, acknowledging that LOBPCG has overwhelming advantages that justify choosing it in practice\footnote{It is a mature part of PETSc, and has implementations in JAX, Torch, and various other languages, and offers the use of preconditioners which are natural for the problem at hand.}. 

\begin{prop}\label{thm:algoUniversalBasis}
	Algorithm $\mathtt{UniversalBasis}$ has runtime $O\( k n_1 + k^2 n_0)/\sqrt{\bar{\gre}} \)$.
\end{prop}

\begin{remark}\label{rmk:noSVD}
	The work of \cite{Lazysvd16} does not require the matrix to be symmetric. So the runtime stated above applies just as well to $\dlap$. The reason for reducing the problem to a graph-laplacian problem is that common downstream tasks that otherwise rely on an Eigenvalue or Singular Value Decomposition may be done faster on graphs without explicitly computing this decomposition. In the case where this, often significant, computational advantage is not sought from graphs, $\mathtt{UniversalBasis}$ may instead be defined to return the corresponding $k$ eigenvectors of $\dstar L$ instead of $L$.
\end{remark}

Some examples of the advantage of working with graph-laplacians include the speedup of spectral clustering using polynomial graph filters \cite{KnyazevPoly15}, or chebyshev filters for $k$-means \cite{ChebyPolyFilterClustering24}. Other examples include the quadratic-form preserving sparsification of graphs (or introduction of variance-preserving independence between components), pioneered by the works of \cite{SpecSparsSpielEtAl13}, which has been recently proven to have an especially fast and simple equivalent in certain cases \cite{PetrosUniformEdge25}.  

Nevertheless, we provide one immediate consequence of a low-rank approximation to $C^{\gra, \grk}$ - a low-dimensional representation of a GMRF: 

\begin{minimalgo}{$\mathtt{SquashGMRF}$ - Low-Dimensional Representation of GMRF}
	\textbf{Assumes:} $\mathtt{UniversalBasis}$ to provide a basis guaranteeing an projection error, $\gre$ \\ 
	\textbf{Input:} \begin{enumerate}
		\item[\quad] $(X_0, X_1, \dstar_1)$ for $X \in \{M, M'\}$;
		\item[\quad] Relative-error, $\gre > \grr_M \eta$; 
		\item[\quad] Measurements of a GMRF, $Y \in \reals{n_0, N}$, $X_0$.
	\end{enumerate}
	\textbf{Returns: $Z \in \reals{kN}, k \ll N$.}

	\begin{enumerate}[label=\arabic*., leftmargin=1.5em, itemsep=2pt, topsep=8pt]
		\item $\tilde{V}_k \from \mathtt{UniversalBasis}(X_0, X_1, \dstar_1), \gre)$ 
		\item Return $\tilde{V}_k ^{\top} Y$ \comment{$O(k n_0 N)$}
	\end{enumerate}
\end{minimalgo}

\begin{enumerate}
	\item $Z = \mathtt{SquashGMRF}((X_0, X_1, \dstar_1), \gre, Y)$, may be used, for instance, in learning $\gra, \grk$ from multiple measurements of a GMRF, by modelling the decay of the spectrum of $Z^{\top} Z \in \reals{k^2}$. 
	\item Since $\mathtt{UniversalBasis}$ returns a small basis for an arbitrarily large problem, we have the natural compressed-sensing result, that $O(k \log(k))$ measurements in $Y$ suffice to preserve the output of $\mathtt{SquashGMRF}$\footnote{This shall be made more rigorous in later versions, in the interest of brevity}. In applications where measurements are of highly varying quality/reliability/difficulty, this reduces error-propagation by allowing for a low number of high-quality measurements to be modelled on. 
\end{enumerate}.

%% file: assume.laplacian_incoherence.tex
\begin{assume}\label{ass:laplacian_incoherence}
	Let $\bar{L}_h := \dd_{m-1} \bar{\dstar} _1 \bar{\dd}_0 = \bar{\dd}_0 ^{\top} \bar{\dstar} _1 \bar{\dd} _0$ defined over the polytopal $M' _h$. Let $\bar{L}_h$ have the property that for all $0 \leq h< h^*$, its eigenspace, $S_{h, k}$, associated with its smallest $k$ positive eigenvalues, has an $l_{\infty}$ bounded $l_2$ unit-sphere: $\forall x \in S_{h, k}, \norm{x}_2 = 1$,
	\[ \norm{x}_{\infty} = \max_i \abs{x_i} = O(\sqrt{1/\bar{n}_h}), \]
	where $\bar{n}_h$ is the number of vertices in $M' _h$. 
\end{assume}

%% file: lma.li_bound.tex
\begin{lemma}\label{lma:li_bound}
	Let $L \in \reals{n, n}$ be a graph-laplacian of a connected graph on $n$ vertices. Let $D$ be a diagonal matrix close to $I_n$, the identity. Let $V_k, \tilde{V}_k$ be orthonormal bases for $k$-dimensional right singular (row) spaces $S_k, \tilde{S}_k$ corresponding to the smallest $k$ positive singular values of $DL, L$ respectively.
	Let $\sin(S_k, \tilde{S}_k)$ be the diagonal matrix of the sines of principal angles between $S_k, \tilde{S}_k$. Then for $V_k$, the eigenbasis of $L$ spanning $S_k$,
	\begin{equation} \label{eqn:li_bound} 
		\norm{\sin(S_k, \tilde{S}_k)} = \norm{\sin(\tilde{S}_k, S_k)} = k \sqrt{\norm{(I - D)V_k}^2 + \norm{(I - \inv{D})V_k}^2}.
	\end{equation}
\end{lemma}

%% file: thm.eigspace_approx.tex
\begin{thm} \label{thm:eigspace_approx}
	 Let $\{\abs{\grs_i}\}$ be the set of absolute volumes of $\grs_i \in M$ with mean $\grm_m$, and the property that 
\begin{equation}\label{constr:vol_bounded} 
	1/(1+\gre_s) \leq \abs{\grs_i}/\grm_m \leq (1 + \gre_s), \;\; 0 \leq \gre_s < 1/2.
\end{equation} 

\begin{equation}\label{constr:vol_concentrated}  
	\( \sum_i \( \abs{\grs_i}/\grm_m - 1 \)^2 \)/n \leq s^2.
\end{equation} 

	Let $S_k, \tilde{S}_k$ be the $k$-dimensional eigenspaces corresponding to the $k$ smallest eigenvalues of $\bar{\dlap}_0, \bar{L}$ respectively, with $\grP_k, \tilde{\grP}_k$ being the orthogonal projections onto these subspaces respectively. Then
\begin{equation}\label{result:proj_err}
	\norm{\grP_k - \tilde{\grP}_k}_2 = O(ks).
\end{equation} 
\end{thm}

We call the constraint, (\ref{constr:vol_concentrated}), \emph{relative volume-concentration}. We observe that (\ref{result:proj_err}) is meaningful only when $s \ll 1/k$. We call $M$ \emph{relatively volume-concentrated} if this is the case. More formally,
\begin{defn}\label{def:rvc}
	A uniform $m$-dimensional simplicial complex, $M$, is said to be $\gre$ relatively volume-concentrated or to have the property of $\gre$ relative volume-concentration ($\gre$-RVC) if the sample-variance of 
	\[ z_i = n_m \abs{\grs_i}/ \sum_j \abs{\grs_j}   \]
	across all $n_m$ maximal simplices, $\grs_i \in M$, is no greater than $\gre^2 < 1$.
\end{defn}
\begin{proof}
	We begin by noting that $S_k, \tilde{S}_k$ are invariant to positive scalings of $\dlap, L$ and so we may scale $\dlap$ by $\grm_m$. Since $(\grm_m  \dstar)(i, i) = \grm_m/\abs{\grs_i}$ we have by (\ref{constr:vol_bounded}):
	\begin{equation}\label{constr:inv_vol_bounded} 
		\frac{1}{(1 + \gre_s)} \Leq \grm_m \dstar \Leq 1+\gre_s, \;\;  \frac{1}{(1 + \gre_s)} \Leq \frac{\inv{\dstar}}{\grm_m} \Leq 1+\gre_s.
\end{equation}
If $V_{k, \perp}, \tilde{V}_{k, \perp}$ are orthonormal bases for the subspaces orthogonal to $S_k, \tilde{S}_k$ respectively, then  
\[ \norm{\sin(V_k, \tilde{V}_k)} = \norm{V_k ^{\top} \tilde{V}_{k, \perp}} = \norm{V_k V_k ^{\top} \tilde{V}_{k, \perp} \tilde{V}_{k, \perp} ^{\top}} = \norm{\grP_k (I - \tilde{\grP}_k)}   \]
from unitary invariance of $\norm{\cdot}$. We note that for $\gre < 1/2$,
	\begin{equation}\label{eqn:normalProj2ProjErr}
		\norm{(I - \grP_k) \tilde{\grP}_k }_2 = O(\gre) \iff \norm{\grP_k - \tilde{\grP}_k }_2 = O(\gre)  
	\end{equation}
	Lemma \ref{lma:li_bound} with $D =\grm_m \dstar$, $E := I - \grm_m \dstar$, $E' := I - \inv{\dstar}/\grm_m$ now yields: 
\begin{equation}\label{eqn:li_bound_dec} 
	\norm{\sin(V_k, \tilde{V}_k)}_2 = \norm{(I - \grP_k) \tilde{\grP}_k} \leq k \sqrt{\norm{E V_k}_2 ^2 + \norm{E' V_k}_2 ^2}.
\end{equation}

Now $\norm{E'}_F ^2 / n \leq s^2$ by (\ref{constr:vol_concentrated}) and $\norm{E}_F ^2 /n = s^2 O(1)$ since $E \Leq \gre_s$ by (\ref{constr:inv_vol_bounded}) whence, with 
\[ y := \displaystyle \arg \max_{x \in S_k, \norm{x}=1} (x, E^2x), \]
we have that
\[\norm{E V_k}_2 ^2 = (y, E^2 y) = \sum_i E_{ii} ^2 y_i ^2  \leq \norm{y}_{\infty} ^2 \sum E_{ii} ^2  = \norm{y}_{\infty} ^2 n s^2 O(1). 
\]
	Since $S_k$ is a bottom eigenspace of $L$, a graph-laplacian satisfying Assumption \ref{ass:laplacian_incoherence}, we have $\norm{y}_{\infty} ^2 = O(1 / n)$. So $\norm{EV_k}_2 ^2, \norm{E' V_k}_2 ^2$ are $O(s^2)$, immediately yielding (\ref{result:proj_err}) from (\ref{eqn:li_bound_dec}) via (\ref{eqn:normalProj2ProjErr}).
\end{proof}

%% file: S.Results.tex
\section{Experiment and Results} \label{sec:numResults}
\begin{figure}[ht] \label{fig:earth_4572}
	\centering
	\includegraphics[scale=0.34]{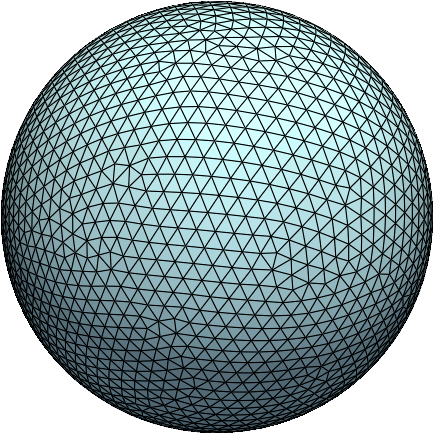}
	\caption{Mesh of the smooth earth's surface used in this experiment. Created with $\mathtt{JIGSAW}$ \cite{jigsaw1}} 
\end{figure}

\begin{figure}[h] \label{fig:subspApprox}
	\includegraphics[scale=0.37]{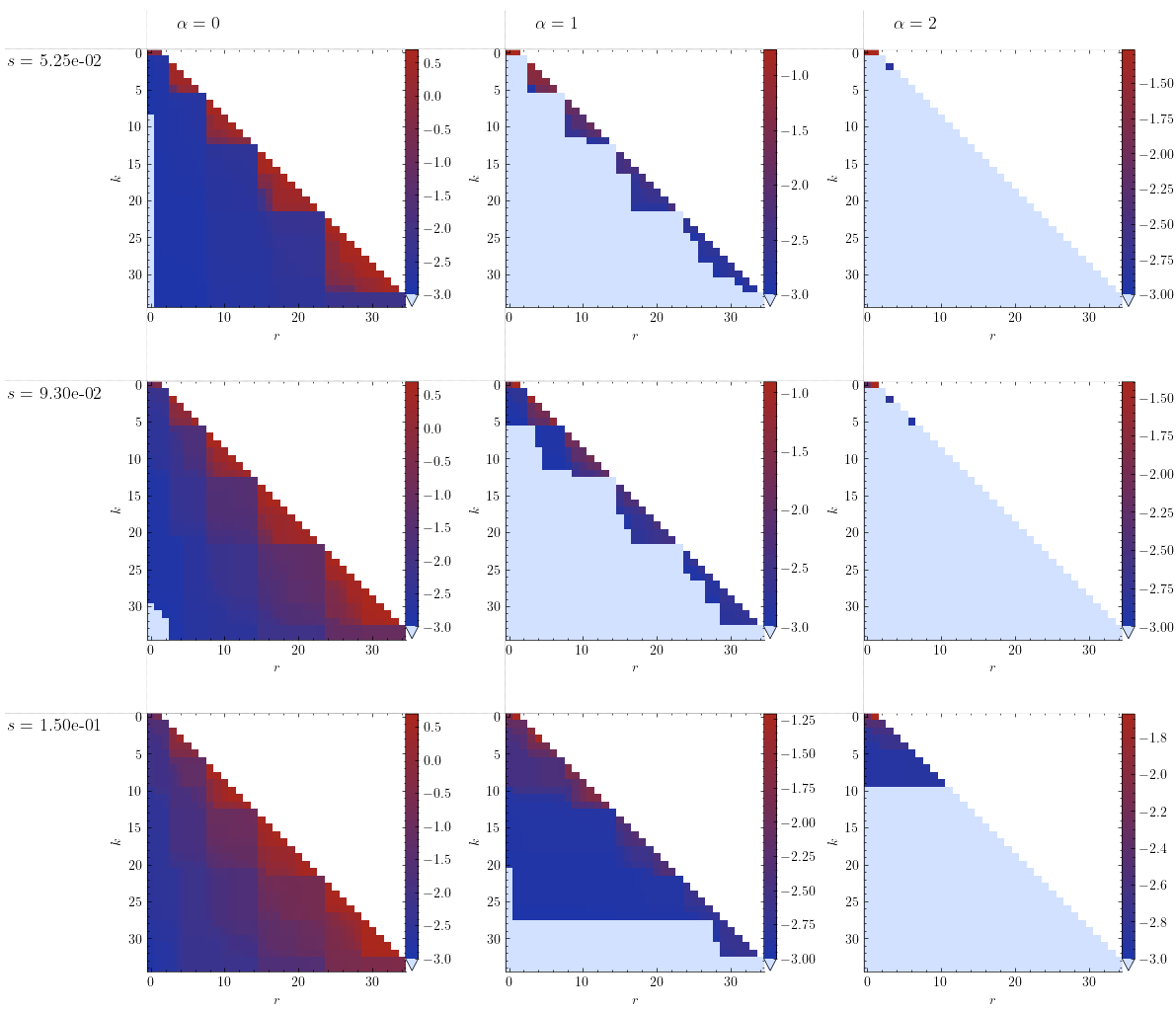}
	\caption{Relative Errors (log) Approximating the $r$-rank Covariance, $\bar{C}_r ^{\gra}$, by Projection, $\tilde \grP_k$, onto the Lowest $k$-dimensional Eigenspace of $\bar{L} = \dd \dstar \dd$ across Relative Volume Deviations, $s$. The top left of each panel corresponds to $(0, 0)$ and $r, k$ increase toward the bottom right in each panel; $(\gra, s)$ do the same over the whole grid.}
\end{figure}

We present the results for the representative example of a mildly oblate spheroid - the earth - remeshed from MPAS\footnote{\url{https://mpas-dev.github.io/atmosphere/atmosphere_meshes.html}} to have $4572$ triangles\footnote{Further experiments conducted using an ellipsoid may be added to later versions of this, that detail how submeshing may be done to achieve volume-concentration in cases where the mesh has a high variance in absolute $m$-plicial volumes. This is avoided here in the interest of brevity}. The goal of this numerical experiment is to see how the errors from approximating the low-rank-truncation, $C^{\gra} _r = \grP_r C^{\gra}$, of the covariance matrix, $C^{\gra}$, of the \emph{simplicial mean} of the GMRF solving (cite eqn) changes with $r$ and the rank, $k$, of the projection, $\tilde{\grP}$. Note that $\grk$ is kept fixed; its effects for this particular goal are considered secondary.

\subsection{Experimental Method}
The MPAS surface-mesh re-triangulated into $4572$ nearly equilateral triangles is used as in Fig. \ref{fig:earth_4572}. Volume deviation is introduced by \emph{latitude-spacing-fields} that vary more toward the equator. Each such latitude-spacing-field corresponds to a single perturbation of the volumes of the triangles in the mesh, resulting in a particular relative-volume-deviation, $s$ as in (\ref{thm:main}). The covariance matrix, $C^{\gra}$, is obtained for various values of $\gra$, three of which are depicted below. All perturbed meshes are kept fully well-centered for the results shown below.

An approximation error, $\gre(r, k)$, is computed for each pair $(r, k)$:
\[ \gre(r, k) = \frac{\norm{\tilde{\grP}_k - \grP_r}_2}{\norm{C^{\gra}}_2}.  \]
 This is consistent with Theorem \ref{thm:main} but for one exception: $r \leq k$ in the experiment instead of $r=k$ as in Theorem \ref{thm:main}. $\log(\gre(r, k))$ is plotted as a heatmap in Figure \ref{fig:subspApprox}. 

\subsection{Results and Discussion}
$\gra=0$ corresponds to the white noise\footnote{verified by observing that exponentiating by 0 gives a spectrum that is the identity, making the resulting covariance the identity} and serves to present the relative alignment of $S_r, \tilde{S}_r$ alone as a baseline. In each column, the errors are monotonic increasing with $s$. In each row, going left to right, the increased spectral decay of $C^{\gra}$ with increase in $\gra$ quickly reduces the approximation errors to below the threshold of $1e-3$. The errors tend to be appreciable only near the diagonal of each panel, where $r = k$, indicating that the laplacian eigenbasis is \emph{incrementally} matched to the eigenbasis of $C^{\gra}$: that the leading eigenspace of the graph-laplacian, $\bar{L}$, projects mostly into the leading eigenspace of $C^{\gra}$ and its projection into lower order eigenspaces of $C^{\gra}$ sharply decreases \footnote{Here "eigenspace" is used to mean the space spanned by eigenvectors having the same eigenvalue; eigenvalues ordered decreasingly}. Critically, the approximation errors get below the tolerance of $1e-3$, for a relatively small $k$: in each panel $k, r$ are at most $35$ in a $4572$-triangle-mesh (less than $1\%$), with $k<30$ sufficing for $\gra \geq 1$. This is consistent with the increased stability of the bottom eigenspaces to relatively incoherent perturbations to a matrix. Since the hodge-star, $\dstar_m$, is seen as a multiplicative perturbation in volume-concentrated meshes, the relatively stable lower eigenspaces of $\bar{P}$ are close to those of $L$, confirming the behaviour near the panel-diagonals.

%% file: S.FutureWork.tex
\section{Conclusion and Future Work}\label{sec:futurework}
Here we developed a convergent approximation to the covariance of the family of continuous Gaussian processes over many manifolds of arbitrary dimension, where the discrete representation meaningfully stands for the covariance of measurements made of the continuous process. Further, we developed an efficient algorithm to obtain a single projection that low-rank approximates the entire two-parameter family of discrete GMRFs arising from Mat\'ern fields, and provided two immediate use-cases for such an approximation: parameter-inference, and compressed-sensing. Along the way, we observed that Discrete Exterior Calculus was consistent with the approach of the Virtual Element Method.

Formally characterizing manifolds for which Assumption \ref{ass:laplacian_incoherence} holds, ensuring that the volume-concentration property may be arrived at easily using submeshing, and studying the role of curvature on the results, have been withheld from this initial version. 

The future work following Remark \ref{rmk:noSVD} is to further investigate archetypal applications of the GMRF-approximation to determine if and how the recent advancements in graph-based numerical algorithms circumvent computing the basis, $V$, explicitly. 

The future work in the direction of mathematical efforts here is to study the nature of vector-GMRFs - which model winds, ocean currents, etc. - as well as higher order forms, where the DEC hodge-laplacian ceases to have a graph-laplacian "core" and becomes more dense. This has already been undertaken in the continuous case in works such as \cite{NicoudKrause24,CaoSheffieldGausskForms26,Borovitskiy20}, and further endeavour here seems likely to be fruitful. 

The underlying thematic question, \emph{"How does one summarize a manifold?"}, in the sense of holding enough information to convey its topological and metric features either probabilistically or deterministically to be of computational/mathematical use, is of great interest to the author.

%% file: S.Credits.tex
\section{Credits}
The computations were mostly performed on Google $\mathtt{colab}$, using $\mathtt{JAX}$. Darren Engwirda's $\mathtt{JIGSAWPY}$ was used for meshing and submeshing \footnote{in results that have been left out of this initial version}.  

\subsection{AI-Use} $\mathtt{Gemini}$ was used for searches\footnote{The default behaviour for Google searches involves AI-response. The author used it directly for quicker, higher quality results}, for finding any missing literature, for citations, and for help when learning $\mathtt{JAX}$; for these purposes, the author found it phenomenally useful. 

No part of this writeup was written/paraphrased from any AI-response. No part of the code used for this work was produced by AI at the request of the author\footnote{Separate notebooks were used for $\mathtt{JAX}$ snippets answered by $\mathtt{Gemini}$ for tinkering to learn}. Individuals and communities are comfortable with various extremes of AI use; the author finds it healthy to disclose the nature of its use in peer-reviewed work.

\section{Acknowledgement}
The author thanks everyone who, in some way, directly caused circumstances where this problem could be worked on fruitfully. Mike Davies, Istvan Gy\"ongy, Peter Kramer, Malik Magdon-Ismail, Endre S\"uli, Abiy Tasissa all deserve thanks. As do the various mathematics libraries in France and Sweden, and the National Library of Scotland in Edinburgh. The manifold hills of Edinburgh, The Pentlands and The Cairngorms were all particularly helpful.

%% file: thebib.bib
@article{rue11,
    author = {Lindgren, Finn and Rue, Håvard and Lindström, Johan},
    title = {An Explicit Link between Gaussian Fields and Gaussian Markov Random Fields: The Stochastic Partial Differential Equation Approach},
    journal = {Journal of the Royal Statistical Society Series B: Statistical Methodology},
    volume = {73},
    number = {4},
    pages = {423-498},
    year = {2011},
    month = {08},
    issn = {1369-7412},
    doi = {10.1111/j.1467-9868.2011.00777.x},
    url = {https://doi.org/10.1111/j.1467-9868.2011.00777.x},
    eprint = {https://academic.oup.com/jrsssb/article-pdf/73/4/423/49163491/jrsssb_73_4_423.pdf},
}

@article{rue22,
title = {The SPDE approach for Gaussian and non-Gaussian fields: 10 years and still running},
journal = {Spatial Statistics},
volume = {50},
pages = {100599},
year = {2022},
note = {Special Issue: The Impact of Spatial Statistics},
issn = {2211-6753},
doi = {https://doi.org/10.1016/j.spasta.2022.100599},
url = {https://www.sciencedirect.com/science/article/pii/S2211675322000057},
author = {Finn Lindgren and David Bolin and Håvard Rue}
}

@article{RueINLA09,
  title = {Approximate Bayesian inference for latent Gaussian models by using integrated nested Laplace approximations},
  author = {Rue, Håvard and Martino, Sara and Chopin, Nicolas},
  journal = {Journal of the Royal Statistical Society: Series B (Statistical Methodology)},
  volume = {71},
  number = {2},
  pages = {319--392},
  year = {2009},
  publisher = {Wiley Online Library}
}

@article{LindgrenRINLA15,
  title = {Bayesian spatial modelling with {R-INLA}},
  author = {Lindgren, Finn and Rue, Håvard},
  journal = {Journal of Statistical Software},
  volume = {63},
  number = {19},
  pages = {1--25},
  year = {2015},
  url = {https://www.jstatsoft.org/v63/i19/}
}

@article{bachl19inlabru,
  title = {inlabru: an R package for Bayesian spatial modelling from ecological survey data},
  author = {Bachl, Fabian E. and Lindgren, Finn and Borchers, David L. and Illian, Janine B.},
  journal = {Methods in Ecology and Evolution},
  volume = {10},
  number = {6},
  pages = {760--766},
  year = {2019},
  publisher = {Wiley Online Library},
  doi = {10.1111/2041-210X.13168},
  url = {https://besjournals.onlinelibrary.wiley.com/doi/abs/10.1111/2041-210X.13168}
}

@article{KellerLindstrom14,
 ISSN = {19326157, 19417330},
 URL = {http://www.jstor.org/stable/23024839},
 author = {David Bolin and Finn Lindgren},
 journal = {The Annals of Applied Statistics},
 number = {1},
 pages = {523--550},
 publisher = {Institute of Mathematical Statistics},
 title = {SPATIAL MODELS GENERATED BY NESTED STOCHASTIC PARTIAL DIFFERENTIAL EQUATIONS, WITH AN APPLICATION TO GLOBAL OZONE MAPPING},
 urldate = {2026-05-31},
 volume = {5},
 year = {2011}
}

@article{OzoneBolinLindgren11,
 ISSN = {19326157, 19417330},
 URL = {http://www.jstor.org/stable/23024839},
 author = {David Bolin and Finn Lindgren},
 journal = {The Annals of Applied Statistics},
 number = {1},
 pages = {523--550},
 publisher = {Institute of Mathematical Statistics},
 title = {SPATIAL MODELS GENERATED BY NESTED STOCHASTIC PARTIAL DIFFERENTIAL EQUATIONS, WITH AN APPLICATION TO GLOBAL OZONE MAPPING},
 urldate = {2026-05-31},
 volume = {5},
 year = {2011}
}

@article{OceanTempDinsdale19,
  title={Modelling ocean temperatures from bio-probes under preferential sampling},
  author={Dinsdale, Daniel and Salibi{\'a}n-Barrera, Mat{\'\i}as},
  journal={The Annals of Applied Statistics},
  volume={13},
  number={2},
  pages={713--745},
  year={2019},
  publisher={Institute of Mathematical Statistics}
}

@article{Konstantinoudis20DiseaseModeling,
  title={Discrete versus continuous domain models for disease mapping},
  author={Konstantinoudis, Garyfallos and Schuhmacher, Dominic and Rue, H{\aa}vard and Spycher, Ben D},
  journal={Spatial and Spatio-temporal Epidemiology},
  volume={32},
  pages={100319},
  year={2020},
  publisher={Elsevier},
  doi={10.1016/j.sste.2019.100319},
  url={https://doi.org/10.1016/j.sste.2019.100319}
}

@article{MalariaBhatt15,
  title={The effect of malaria control on Plasmodium falciparum in Africa between 2000 and 2015},
  author={Bhatt, Samir and Weiss, Daniel J and Cameron, Ewan and Bisanzio, Donal and Mappin, Bob and Dalrymple, Ursula and Battle, Katherine E and Moyes, Catherine L and Henry, Amanda and Eckhoff, Philip A and others},
  journal={Nature},
  volume={526},
  number={7572},
  pages={207--211},
  year={2015},
  publisher={Nature Publishing Group}
}

@article{Cameletti2013Spatio,
  title={Spatio-temporal modeling of particulate matter concentration through the SPDE approach},
  author={Cameletti, Michela and Lindgren, Finn and Simpson, Daniel and Rue, H{\aa}vard},
  journal={AStA Advances in Statistical Analysis},
  volume={97},
  number={2},
  pages={109--131},
  year={2013},
  publisher={Springer}
  }

@article{Fuglstad15,
  title={Exploring a new class of non-stationary spatial Gaussian random fields with varying local anisotropy},
  author={Fuglstad, Geir-Arne and Lindgren, Finn and Simpson, Daniel and Rue, H{\aa}vard},
  journal={Statistica Sinica},
  volume={25},
  number={1},
  pages={115--133},
  year={2015},
  publisher={Institute of Statistical Science, Academia Sinica}
}

@article{RueSimpsonThink12,
    title = {Think continuous: Markovian Gaussian models in spatial statistics},
    journal = {Spatial Statistics},
    volume = {1},
    pages = {16-29},
    year = {2012},
    issn = {2211-6753},
    doi = {https://doi.org/10.1016/j.spasta.2012.02.003},
    publisher={Elsevier},
    author = {Daniel Simpson and Finn Lindgren and Håvard Rue}
}

@phdthesis{Hirani03,
  author = {Hirani, Anil Nirmal},
  title = {Discrete Exterior Calculus},
  school = {California Institute of Technology},
  year = {2003},
  month = {May},
  address = {Pasadena, CA, USA},
  note = {UMI Order Number: AAT 3081691}
}

@article{Desbrun03,
  author = {Desbrun, Mathieu and Hirani, Anil N. and Leok, Melvin and Marsden, Jerrold E.},
  title = {Discrete Exterior Calculus},
  journal = {arXiv preprint},
  year = {2003},
  volume = {cs/0310014}
}

@phdthesis{Kalyanaraman16Hodge,
  author       = {Kalyanaraman, Koushik},
  title        = {Hodge Laplacians on simplicial meshes and graphs},
  school       = {University of Illinois at Urbana-Champaign},
  year         = {2016},
  address      = {Urbana, IL, USA},
  url          = {https://www.ideals.illinois.edu/items/91250},
  note         = {Available via IDEALS}
}

@article{LHL,
	author = {Lim, Lek-Heng},
	title = {Hodge Laplacians on Graphs},
	journal = {SIAM Review},
	volume = {62},
	number = {3},
	pages = {685-715},
	year = {2020},
	doi = {10.1137/18M1223101},
	URL = {https://doi.org/10.1137/18M1223101},
	eprint = {https://doi.org/10.1137/18M1223101}
}

@article{AFW06,
author = {Arnold, Douglas and Falk, Richard and Winther, Ragnar},
year = {2006},
month = {05},
pages = {1 - 155},
title = {Finite element exterior calculus, homological techniques, and applications},
volume = {15},
journal = {Acta Numerica},
doi = {10.1017/S0962492906210018}
}

@article{AFW09,
author = {Arnold, Douglas and Falk, Richard and Winther, Ragnar},
year = {2009},
month = {06},
pages = {},
title = {Finite element exterior calculus: From Hodge theory to numerical stability},
volume = {47},
journal = {Bulletin of the American Mathematical Society},
doi = {10.1090/S0273-0979-10-01278-4}
}

@book{AFWBook18,
  title={Finite Element Exterior Calculus},
  author={Arnold, Douglas N. and Falk, Richard S. and Winther, Ragnar},
  year={2018},
  publisher={Society for Industrial and Applied Mathematics},
  address={Philadelphia, PA}
}

@article{whittle1963,
  title={Stochastic processes in several dimensions},
  author={Whittle, Peter},
  journal={Bulletin of the International Statistical Institute},
  volume={40},
  pages={974--994},
  year={1963}
}

@misc{Kaibo26,
      title={Convergence and Stability of Discrete Exterior Calculus for the Hodge Laplace Problem in Two Dimensions}, 
      author={Chengbin Zhu and Snorre H. Christiansen and Kaibo Hu and Anil N. Hirani},
      year={2026},
      eprint={2505.08966},
      archivePrefix={arXiv},
      primaryClass={math.NA},
      url={https://arxiv.org/abs/2505.08966}, 
}

@misc{GP26,
      title={A Framework for Analysis of DEC Approximations to Hodge-Laplacian Problems using Generalized Whitney Forms}, 
      author={Johnny Guzmán and Pratyush Potu},
      year={2026},
      eprint={2505.08934},
      archivePrefix={arXiv},
      primaryClass={math.NA},
      url={https://arxiv.org/abs/2505.08934}, 
}

@article{TS16DEC,
  title     = {Convergence of Discrete Exterior Calculus Approximations for the Poisson Equation},
  author    = {Tsogtgerel, Gantumur and Schultz, Erick},
  journal   = {arXiv preprint arXiv:1611.03955},
  year      = {2016},
  url       = {https://arxiv.org/abs/1611.03955}
}

@book{whitney57geoint,
  title={Geometric Integration Theory},
  author={Whitney, Hassler},
  year={1957},
  publisher={Princeton University Press},
  address={Princeton, NJ}
}

@book{AMR88,
  title     = {Manifolds, Tensor Analysis, and Applications},
  author    = {Abraham, Ralph and Marsden, Jerrold E. and Ratiu, Tudor},
  edition   = {2nd},
  series    = {Applied Mathematical Sciences},
  volume    = {75},
  year      = {1988},
  publisher = {Springer-Verlag},
  address   = {New York, NY},
  isbn      = {978-0-387-96790-5},
  doi       = {10.1007/978-1-4612-1029-0}
}

@techreport{liRelPert94,
  author      = {Li, Ren-Cang},
  title       = {Relative Perturbation Theory: {(II)} {Eigenspace} and Singular Subspace Variations},
  institution = {Computer Science Division, University of California, Berkeley},
  number      = {UCB/CSD-94-856},
  year        = {1994},
  note        = {Also LAPACK Working Note 85}
}

@misc{CraneDDG,
    author = {Crane, Keenan},
    title = {Discrete Differential Geometry},
    howpublished = {https://brickisland.net/ddg-web/},
    year = {2018},
    note = {Course notes available at Carnegie Mellon University}
}

@incollection{CraneEC23,
    author = {Crane, Keenan},
    title = {Exterior Calculus in Graphics: Course Notes},
    booktitle = {ACM SIGGRAPH 2023 Courses},
    year = {2023},
    publisher = {ACM},
    doi = {10.1145/3587423.3595525},
    url = {https://dl.acm.org/doi/10.1145/3587423.3595525}
}

@article{GeoDistPredCrane17,
	author = {Crane, Keenan and Weischedel, Clarisse and Wardetzky, Max},
	title = {The heat method for distance computation},
	year = {2017},
	issue_date = {November 2017},
	publisher = {Association for Computing Machinery},
	address = {New York, NY, USA},
	volume = {60},
	number = {11},
	issn = {0001-0782},
	url = {https://doi.org/10.1145/3131280},
	doi = {10.1145/3131280},
	journal = {Commun. ACM},
	month = oct,
	pages = {90–99},
	numpages = {10}
}

@inproceedings{kelner09higher,
  title={Higher eigenvalues of graphs},
  author={Kelner, Jonathan A and Lee, James R and Price, Gregory N and Teng, Shang-Hua},
  booktitle={Proceedings of the 50th Annual IEEE Symposium on Foundations of Computer Science (FOCS)},
  pages={735--744},
  year={2009},
  organization={IEEE}
}

@inproceedings{DiscShellsGrinspunHirani03,
	author = {Grinspun, Eitan and Hirani, Anil N. and Desbrun, Mathieu and Schr\"{o}der, Peter},
	title = {Discrete shells},
	year = {2003},
	isbn = {1581136595},
	publisher = {Eurographics Association},
	address = {Goslar, DEU},
	booktitle = {Proceedings of the 2003 ACM SIGGRAPH/Eurographics Symposium on Computer Animation},
	pages = {62–67},
	numpages = {6},
	location = {San Diego, California},
	series = {SCA '03}
}

@article{HodgeRank11,
  author    = {Jiang, Xiaoye and Lim, Lek-Heng and Yao, Yuan and Ye, Yinyu},
  title     = {Statistical ranking and combinatorial {H}odge theory},
  journal   = {Mathematical Programming},
  volume    = {127},
  number    = {1},
  pages     = {203--244},
  year      = {2011},
  publisher = {Springer},
  doi       = {10.1007/s10107-010-0419-x},
  url       = {https://doi.org/10.1007/s10107-010-0419-x}
}

@article{SensNetTahbaz10,
  author    = {Tahbaz-Salehi, Alireza and Jadbabaie, Ali},
  title     = {Distributed coverage verification in sensor networks without location information},
  journal   = {IEEE Transactions on Automatic Control},
  volume    = {55},
  number    = {8},
  pages     = {1837--1849},
  year      = {2010},
  publisher = {IEEE},
  doi       = {10.1109/TAC.2010.2047541},
  url       = {https://doi.org/10.1109/TAC.2010.2047541}
}

@inproceedings{BrainNetHodge19,
	title = {Harmonic holes as the submodules of brain network and network dissimilarity},
	author = {Hyekyoung Lee and Chung, Moo K. and Hongyoon Choi and Hyejin Kang and Seunggyun Ha and Kim, Yu Kyeong and Lee, Dong Soo},
	year = {2019},
	doi = {10.1007/978-3-030-10828-1\_9},
	language = {English},
	isbn = {9783030108274},
	series = {Lecture Notes in Computer Science (including subseries Lecture Notes in Artificial Intelligence and Lecture Notes in Bioinformatics)},
	publisher = {Springer Verlag},
	pages = {110--122},
	editor = {Rebeca Marfil and Mariletty Calder\'on and Antonio Bandera and D\'iaz del R\'io, Fernando and Pedro Real},
	booktitle = {Computational Topology in Image Context - 7th International Workshop, CTIC 2019, Proceedings},
}

@article{LinkPredHodge18,
	author = {Austin R. Benson  and Rediet Abebe  and Michael T. Schaub  and Ali Jadbabaie  and Jon Kleinberg },
	title = {Simplicial closure and higher-order link prediction},
	journal = {Proceedings of the National Academy of Sciences},
	volume = {115},
	number = {48},
	pages = {E11221-E11230},
	year = {2018},
	doi = {10.1073/pnas.1800683115},
	URL = {https://www.pnas.org/doi/abs/10.1073/pnas.1800683115},
	eprint = {https://www.pnas.org/doi/pdf/10.1073/pnas.1800683115}
}

@article{VecFieldDecompGraphics03,
	author = {Tong, Yiying and Lombeyda, Santiago and Hirani, Anil N. and Desbrun, Mathieu},
	title = {Discrete multiscale vector field decomposition},
	year = {2003},
	issue_date = {July 2003},
	publisher = {Association for Computing Machinery},
	address = {New York, NY, USA},
	volume = {22},
	number = {3},
	issn = {0730-0301},
	url = {https://doi.org/10.1145/882262.882290},
	doi = {10.1145/882262.882290},
	journal = {ACM Trans. Graph.},
	month = jul,
	pages = {445–452},
	numpages = {8}
}

@inproceedings{Borovitskiy20,
 	author = {Borovitskiy, Viacheslav and Terenin, Alexander and Mostowsky, Peter and Deisenroth (he/him), Marc},
 	booktitle = {Advances in Neural Information Processing Systems},
 	editor = {H. Larochelle and M. Ranzato and R. Hadsell and M.F. Balcan and H. Lin},
 	pages = {12426--12437},
 	publisher = {Curran Associates, Inc.},
 	title = {Mat\'{e}rn Gaussian Processes on Riemannian Manifolds},
 	url = {https://proceedings.neurips.cc/paper_files/paper/2020/file/92bf5e6240737e0326ea59846a83e076-Paper.pdf},
 	volume = {33},
 	year = {2020}
}

@InProceedings{NicoudKrause24,
	  title = 	 {Intrinsic {G}aussian Vector Fields on Manifolds},
	  author =       {Robert-Nicoud, Daniel and Krause, Andreas and Borovitskiy, Viacheslav},
	  booktitle = 	 {Proceedings of The 27th International Conference on Artificial Intelligence and Statistics},
	  pages = 	 {1306--1314},
	  year = 	 {2024},
	  editor = 	 {Dasgupta, Sanjoy and Mandt, Stephan and Li, Yingzhen},
	  volume = 	 {238},
	  series = 	 {Proceedings of Machine Learning Research},
	  month = 	 {02--04 May},
	  publisher =    {PMLR},
	  url = 	 {https://proceedings.mlr.press/v238/robert-nicoud24a.html}
}

@book{VEMBook22,
  title     = {The Virtual Element Method and its Applications},
  editor    = {Antonietti, Paola F. and Be{\\i}r{\~a}o da Veiga, Louren{\c{c}}o and Manzini, Gianmarco},
  publisher = {Springer International Publishing},
  address   = {Cham},
  year      = {2022},
  volume    = {31},
  series    = {SEMA SIMAI Springer Series},
  isbn      = {978-3-030-95318-8},
  doi       = {10.1007/978-3-030-95319-5},
  url       = {https://link.springer.com/book/10.1007/978-3-030-95319-5}
}

@article{CaoSheffieldGausskForms26,
    author = "Cao, Sky and Sheffield, Scott",
    title = "{Fractional Gaussian Forms and Gauge Theory: An Overview}",
    eprint = "2406.19321",
    archivePrefix = "arXiv",
    primaryClass = "math.PR",
    doi = "10.1007/s11464-024-0135-0",
    journal = "Front. Math.",
    volume = "21",
    number = "1",
    pages = "1--137",
    year = "2026"
}

@article{Rozanov77,
	doi = {10.1070/SM1977v032n04ABEH002404},
	url = {https://doi.org/10.1070/SM1977v032n04ABEH002404},
	year = {1977},
	month = {apr},
	publisher = {},
	volume = {32},
	number = {4},
	pages = {515},
	author = {Ju A Rozanov},
	title = {MARKOV RANDOM FIELDS AND STOCHASTIC PARTIAL DIFFERENTIAL EQUATIONS},
	journal = {Mathematics of the USSR-Sbornik}
	}

@article{PetrosUniformEdge25,
  author       = {Kaiwen He and
                  Petros Drineas and
                  Rajiv Khanna},
  title        = {Structure-Aware Spectral Sparsification via Uniform Edge Sampling},
  journal      = {CoRR},
  volume       = {abs/2510.12669},
  year         = {2025},
  url          = {https://doi.org/10.48550/arXiv.2510.12669},
  doi          = {10.48550/ARXIV.2510.12669},
  eprinttype   = {arXiv},
  eprint       = {2510.12669},
  bibsource    = {dblp computer science bibliography, https://dblp.org}
}

@inproceedings{Lazysvd16,
	author = {Allen-Zhu, Zeyuan and Li, Yuanzhi},
	title = {LazySVD: even faster SVD decomposition yet without agonizing pain},
	year = {2016},
	isbn = {9781510838819},
	publisher = {Curran Associates Inc.},
	address = {Red Hook, NY, USA},
	booktitle = {Proceedings of the 30th International Conference on Neural Information Processing Systems},
	pages = {982–990},
	numpages = {9},
	location = {Barcelona, Spain},
	series = {NIPS'16}
}

@article{GradBoundsGBCFloaterGillette14,
	author = {Floater, Michael and Gillette, Andrew and Sukumar, N.},
	title = {Gradient Bounds for Wachspress Coordinates on Polytopes},
	journal = {SIAM Journal on Numerical Analysis},
	volume = {52},
	number = {1},
	pages = {515-532},
	year = {2014},
	doi = {10.1137/130925712},
	URL = {https://doi.org/10.1137/130925712},
	eprint = {https://doi.org/10.1137/130925712}
}

@article{DiffFormsPolytopeChristiansen08,
	author = {Christiansen, Snorre H.},
	title = {A CONSTRUCTION OF SPACES OF COMPATIBLE DIFFERENTIAL FORMS ON CELLULAR COMPLEXES},
	journal = {Mathematical Models and Methods in Applied Sciences},
	volume = {18},
	number = {05},
	pages = {739-757},
	year = {2008},
	doi = {10.1142/S021820250800284X},
	URL = {
		https://doi.org/10.1142/S021820250800284X
	},
	eprint = {
		https://doi.org/10.1142/S021820250800284X
	}
}

@Article{jigsaw1,
AUTHOR = {Engwirda, D.},
TITLE = {JIGSAW-GEO (1.0): locally orthogonal staggered unstructured grid generation for general circulation modelling on the sphere},
JOURNAL = {Geoscientific Model Development},
VOLUME = {10},
YEAR = {2017},
NUMBER = {6},
PAGES = {2117--2140},
URL = {https://gmd.copernicus.org/articles/10/2117/2017/},
DOI = {10.5194/gmd-10-2117-2017}
}

@article{Floater15, 
	title={Generalized barycentric coordinates and applications}, 
	volume={24}, 
	DOI={10.1017/S0962492914000129}, 
	journal={Acta Numerica}, 
	author={Floater, Michael S.}, 
	year={2015}, 
	pages={161–214}}

@article{HMT11,
	author = {Halko, N. and Martinsson, P. G. and Tropp, J. A.},
	title = {Finding Structure with Randomness: Probabilistic Algorithms for Constructing Approximate Matrix Decompositions},
	journal = {SIAM Review},
	volume = {53},
	number = {2},
	pages = {217-288},
	year = {2011},
	doi = {10.1137/090771806},
	URL = {https://doi.org/10.1137/090771806},
	eprint = {https://doi.org/10.1137/090771806}
}

@article{RST10,
	author = {Rokhlin, Vladimir and Szlam, Arthur and Tygert, Mark},
	title = {A Randomized Algorithm for Principal Component Analysis},
	journal = {SIAM Journal on Matrix Analysis and Applications},
	volume = {31},
	number = {3},
	pages = {1100-1124},
	year = {2010},
	doi = {10.1137/080736417},
	URL = {https://doi.org/10.1137/080736417},
	eprint = {https://doi.org/10.1137/080736417}
}

@article{ChenThom85, 
	title={The lumped mass finite element method for a parabolic problem}, 
	volume={26}, 
	DOI={10.1017/S0334270000004549}, 
	number={3}, 
	journal={The Journal of the Australian Mathematical Society. Series B. Applied Mathematics}, 
	author={Chen, C. M. and Thomée, V.}, 
	year={1985}, 
	pages={329–354}
}

@article{lobpcg1,
  author       = {Andrew V. Knyazev},
  title        = {Toward the Optimal Preconditioned Eigensolver: Locally Optimal Block
                  Preconditioned Conjugate Gradient Method},
  journal      = {{SIAM} J. Sci. Comput.},
  volume       = {23},
  number       = {2},
  pages        = {517--541},
  year         = {2001},
  url          = {https://doi.org/10.1137/S1064827500366124},
  doi          = {10.1137/S1064827500366124},
  bibsource    = {dblp computer science bibliography, https://dblp.org}
}

@article{lobpcg2,
  author       = {Andrew Knyazev},
  title        = {Recent implementations, applications, and extensions of the Locally
                  Optimal Block Preconditioned Conjugate Gradient method {(LOBPCG)}},
  journal      = {CoRR},
  volume       = {abs/1708.08354},
  year         = {2017},
  url          = {http://arxiv.org/abs/1708.08354},
  eprinttype   = {arXiv},
  eprint       = {1708.08354},
  bibsource    = {dblp computer science bibliography, https://dblp.org}
}

@inproceedings{KnyazevPoly15,
  author       = {Andrew Knyazev and
                  Alexander Malyshev},
  editor       = {Deniz Erdogmus and
                  Murat Ak{\c{c}}akaya and
                  Suleyman Serdar Kozat and
                  Jan Larsen},
  title        = {Accelerated graph-based spectral polynomial filters},
  booktitle    = {25th {IEEE} International Workshop on Machine Learning for Signal
                  Processing, {MLSP} 2015, Boston, MA, USA, September 17-20, 2015},
  pages        = {1--6},
  publisher    = {{IEEE}},
  year         = {2015},
  url          = {https://doi.org/10.1109/MLSP.2015.7324315},
  doi          = {10.1109/MLSP.2015.7324315},
  bibsource    = {dblp computer science bibliography, https://dblp.org}
}

@inproceedings{ChebyPolyFIlterClustering24,
  author       = {Liang Du and
                  Yunhui Liang and
                  Mian Ilyas Ahmad and
                  Peng Zhou},
  title        = {K-Means Clustering Based on Chebyshev Polynomial Graph Filtering},
  booktitle    = {{IEEE} International Conference on Acoustics, Speech and Signal Processing,
                  {ICASSP} 2024, Seoul, Republic of Korea, April 14-19, 2024},
  pages        = {7175--7179},
  publisher    = {{IEEE}},
  year         = {2024},
  url          = {https://doi.org/10.1109/ICASSP48485.2024.10446384},
  doi          = {10.1109/ICASSP48485.2024.10446384},
  bibsource    = {dblp computer science bibliography, https://dblp.org}
}

@article{SpecSparsSpielEtAl13,
  author       = {Joshua D. Batson and
                  Daniel A. Spielman and
                  Nikhil Srivastava and
                  Shang{-}Hua Teng},
  title        = {Spectral sparsification of graphs: theory and algorithms},
  journal      = {Commun. {ACM}},
  volume       = {56},
  number       = {8},
  pages        = {87--94},
  year         = {2013},
  url          = {https://doi.org/10.1145/2492007.2492029},
  doi          = {10.1145/2492007.2492029},
  bibsource    = {dblp computer science bibliography, https://dblp.org}
}
